\documentclass[reqno]{amsart}

\usepackage{amsmath}
\usepackage{amssymb}
\usepackage{amsthm} 
\usepackage{graphicx}
\usepackage{caption}
\usepackage{cite}
\usepackage{enumerate}
\usepackage{mathtools}

\newtheorem{Thm}{Theorem}[section]
\newtheorem{Prop}[Thm]{Proposition}
\newtheorem{Lemma}[Thm]{Lemma}
\newtheorem{Cor}[Thm]{Corollary}

\theoremstyle{definition}

\newtheorem{Rem}[Thm]{Remark}

\newtheorem{Coex}[Thm]{Counterexample}

\numberwithin{equation}{section}

\def\bkappa{\boldsymbol\kappa}
\def\btau{\boldsymbol{\tau}}

\def\bw{\mathbf{w}}
\def\bxi{\mathbf{\xi}}
\def\cI{\mathcal{I}}

\def\cS{\mathcal{S}}
\def\Do{\cI}
\def\e{\mathbf{e}}
\def\ints{\int_{\Si}}
\def\intsp{\int_{\Sp}}
\def\intst{\int_{\St}}
\def\N{\mathbb{N}}
\def\nnu{\boldsymbol{\nu}}
\def\pa{\partial}
\def\R{\mathbb{R}}
\def\Si{\Sigma}
\def\Sp{\Sph^1}
\def\Sph{\mathbb{S}}
\def\St{\Si_t}
\def\vare{\varepsilon}
\def\Z{\mathbb{Z}}

\DeclareMathOperator{\cir}{circ}
\DeclareMathOperator{\inner}{in}

\newcommand{\fracd}[2]{{\frac{d #1}{d #2}}}
\newcommand{\fracp}[2]{{\frac{\pa #1}{\pa #2}}}

\DeclareRobustCommand{\SkipTocEntry}[4]{}

\begin{document}

\title{Curve flows with a global forcing term}
\author{Friederike Dittberner}
\address{Department of Mathematics and Statistics, University of Konstanz, Germany}
\email{dittberner@math.fu-berlin.de}

\begin{abstract}
We consider embedded, smooth curves in the plane which are either closed or asymptotic to two lines. 
We study their behaviour under curve shortening flow with a global forcing term.
We prove an analogue to Huisken's distance comparison principle for curve shortening flow for initial curves whose local total curvature does not lie below $-\pi$ and show that this condition is sharp.
With that, we can exclude singularities in finite time for bounded forcing terms.
For immortal flows of closed curves whose forcing terms provide non-vanishing enclosed area and bounded length, we show convexity in finite time and smooth and exponential convergence to a circle.
In particular, all of the above holds for the area preserving curve shortening flow.
\end{abstract}

\maketitle
\tableofcontents

\section{Introduction}

Let $\Si_0\subset\R^2$ be an embedded, smooth curve, parametrised by the embedding $X_0:\Do\to\R^2$, where $\Do\in\{\Sp,\R\}$.
We seek a one-parameter family of maps $X:\Do\times[0,T)\to\R^2$ with $X(\,\cdot\,,0)=X_0$ satisfying the evolution equation
\begin{align}\label{eq:ccf}
\fracp{X}{t}(p,t)=\big(h(t)-\kappa(p,t)\big)\nnu(p,t)
\end{align}
for $(p,t)\in\Do\times(0,T)$, where the vector $\nnu$ is the outward pointing unit normal to the curve $\St:=X(\Do,t)$, $\kappa$ is the curvature function and $T$ is th maximal time of existence.
The global term $h$ is smooth and smoothly bounded whenever the curvature is bounded.
For the curve shortening flow (CSF), $h\equiv0$.
For closed curves, the enclosed area preserving curve shortening flow (APCSF) has the global term
\begin{align}\label{eq:h_ap}
h(t)=\frac{2\pi}{L_t}\,,
\end{align}
where $L_t=L(\St)$ is the length of the curve. 
The length preserving curve flow (LPCF) has the global term
\begin{align}\label{eq:h_lp}
h(t)=\frac1{2\pi}\intst\kappa^2\,ds_t\,.
\end{align}
The total curvature of a curve $\St=X(\Do,t)$ is given by
\begin{align}\label{eq:def_alpha}
\alpha(t):=\intst\kappa\,ds_t\,,
\end{align}
where $\alpha=2\pi$ if the curve $\Si=X(\Sp)$ is embedded, closed and positively oriented.
For $\Do=\R$, we assume that $\Si_t=X_0(\R,t)$ is, up to translation, smoothly asymptotic to two distinct time-independent lines for $p\to-\infty$ and $p\to\infty$, where we also assume that
\begin{align}\label{eq:alpha}
\alpha(t)\equiv\alpha_0\in(-\pi,\pi)
\end{align}
as well as
\begin{align}\label{eq:intabskappa}
\int_{\Si_0}|\kappa|\,ds<\infty\,.
\end{align}
Note that~\eqref{eq:alpha} follows from the evolution equation of $\alpha(t)$, see Lemma~\ref{lem:dsvartheta}, and the asymptotic behaviour of the curve. \\

The APCSF was first studied by Gage~\cite{Gage86}. 
He proved that initially embedded, closed, convex curves stay embedded, smooth and convex, and converge smoothly to a circle of radius $\sqrt{A_0/\pi}$, where $A_0=A(\Si_0)$ is the enclosed area of the initial curve.
In~\cite{Maeder15}, Maeder-Baumdicker studied APCSF for convex curves with Neumann boundary on a convex support curve and showed smooth convergence to an arc for sufficiently short, convex, embedded initial curves. 
She proved a monotonicity formula and excluded type-I singularities for embedded, convex curves under the APCSF.
For the LPCF, Pihan~\cite{Pihan98} showed that initially embedded, closed, convex curves stay embedded, smooth and convex, and converge smoothly and exponentially to a circle of radius $L_0/2\pi$. \\

In this paper, we will adapt theory from CSF.
For CSF in the plane, Gage--Hamilton and Grayson~\cite{GageHamilton86,Grayson87} showed that all embedded, closed initial curves stay embedded until they smoothly and exponentially shrink to a round point. 
In~\cite{Huisken95}, Huisken gave a different proof for this result by bounding the ratio of the extrinsic distance 
$$d(p,q,t):=\Vert X(q,t)-X(p,t)\Vert_{\R^2}$$
and the intrinsic distance
$$l(p,q,t):=\int_p^qds_t$$
for curves $\St=X(\R,t)$ with asymptotic ends, respectively the extrinsic distance and the function
$$\psi(p,q,t):=\frac{L_t}\pi\sin\!\left(\frac{\pi\,l(p,q,t)}{L_t}\right)$$
for curves $\St=X(\Sp,t)$, below away from zero, and by applying singularity theory for CSF.
In~\cite{AndrewsBryan11}, Andrews and Bryan found an explicit function to proof curvature bounds via the distance comparison principle.
To analyse curvature blow-ups, one distinguishes between type-I and type-II singularities and rescales the curve near a point of highest curvature.
Using his famous monotonicity formula in~\cite{Huisken90}, Huisken showed that if an immersed curve develops a type-I singularity under CSF, the curves $\St$ have to be asymptotic to a homothetically shrinking solution around the singular point. 
Abresch and Langer~\cite{AbreschLanger86} had previously classified all embedded, homothetically shrinking solutions of CSF as circles.
One concludes, in case of a type-I singularity, that the curves shrink to a round point.
For the type-II singularities, Hamilton~\cite{Hamilton89} and Altschuler~\cite{Altschuler91} showed that each rescaling sequence converges to a translating solution. 
For curves in the plane, the only solution of this kind is the so-called grim reaper which is, for all $\tau\in\R$, given by the graph of the function $u(\sigma,\tau)=\tau-\log\cos(\sigma)$, where $\sigma\in(-\pi/2,\pi/2)$. 
On the grim reaper $\inf(d/l)=0$, so that type-II singularities can be excluded. 
Since $T<\infty$ and a singularity has to form, it has to be of type~I. \\

This paper is structured as follows.
In Section~\ref{sec:evo}, we state evolution equations for the geometric quantities under~\eqref{eq:ccf} and draw first conclusions. 
In Section~\ref{sec:theta2}, we consider angles of tangent vectors and derive a strong maximum principle for the local total curvature
\begin{align}\label{eq:intkappa}
\theta(p,q,t):=\int_p^q\kappa\,ds_t\,.
\end{align}
In the subsequent sections, we study the flow~\eqref{eq:ccf} for embedded, positively oriented, smooth initial curves $\Si_0=X_0(\Do)$ with
\begin{align}\label{eq:intkappageqminuspi}
\theta_0(p,q)=\int_p^q\kappa\,ds\geq-\pi
\end{align}
for all $p,q\in\Do$. 
Note that for convex curves $\theta_0\geq0$.
Figure~\ref{fig:example} is an example for condition~\eqref{eq:intkappageqminuspi}, where all the angles lay between $-\pi$ and $3\pi$, e.\,g. $\theta(p,q)=-\pi$, $\theta(q,p)=3\pi$, $\theta(q,r)=2\pi$, $\theta(r,q)=0$, $\theta(r,p)=\pi$.
\begin{figure}
\centering
\begin{minipage}{.35\textwidth}
	\centering
	\includegraphics[width=\textwidth]{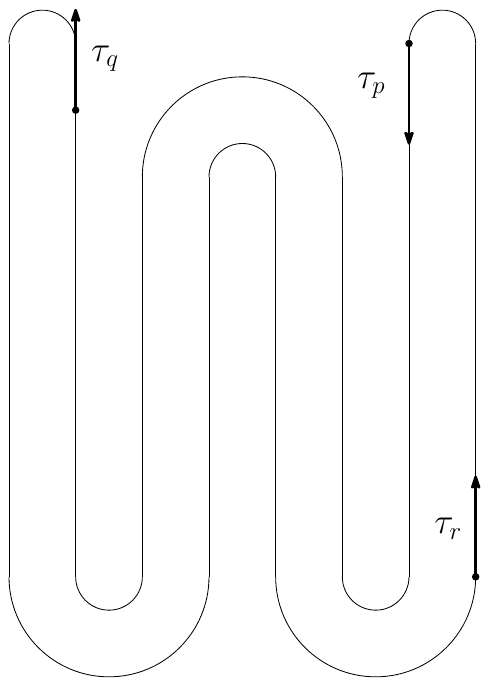}
	\captionsetup{width=.9\textwidth}
	\caption{Example for condition~\eqref{eq:intkappageqminuspi}}
	\label{fig:example}
\end{minipage}
\hfill
\begin{minipage}{.60\textwidth}
	\centering
	\includegraphics[width=\textwidth]{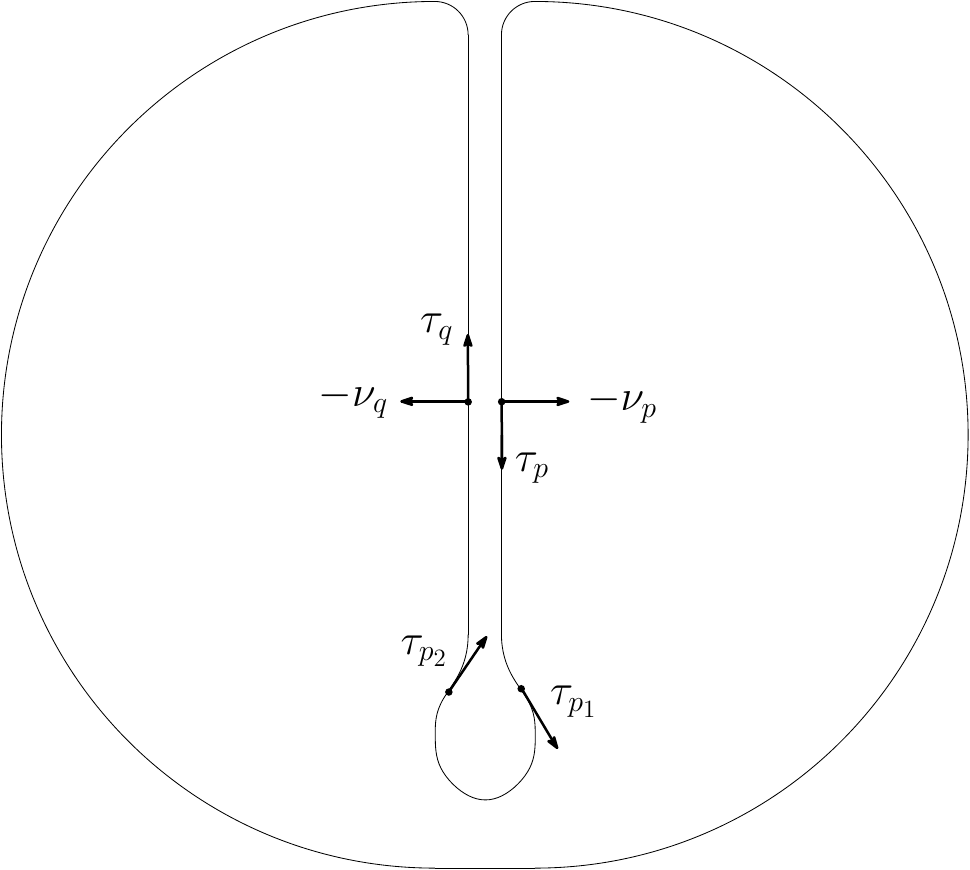}
	\caption{Counterexample~\ref{ex:cexample}}
	\label{fig:cexample}
\end{minipage}
\end{figure}\\

In Sections~\ref{sec:noncompact} and~\ref{sec:noncollapsing}, we modify the distance comparison principles from~\cite{Huisken95} and prove that, for
\begin{align}\label{eq:h_1}
h(t)\in
\begin{dcases}
[0,\infty)&\qquad\text{ for }\Do=\R \\
\left[0,\frac1{2\pi}\intst\kappa^2\,ds_t+\frac{2\pi}{L_t}\right]&
\qquad\text{ for }\Do=\Sp
\end{dcases}
\end{align}
and if the initial embedding $\Si_0$ satisfies~\eqref{eq:intkappageqminuspi}, the ratio $d/l$ for $\Do=\R$ and $d/\psi$ for $\Do=\Sp$ is bounded from below away from zero uniformly in time. 
We conclude that the curves $\St$ stay embedded for all $t\in[0,T)$.
We also show that the condition~\eqref{eq:intkappageqminuspi} is sharp, that is, one can construct initial curves which violate~\eqref{eq:intkappageqminuspi} arbitrarily mildly and for which the resulting flow self-intersects in finite time.
An example is the initial curve in Figure~\ref{fig:cexample} with length sufficiently large compared to the $C^{3,\alpha}$-norm of its embedding and for which $\min_{\Do\times\Do}\theta_0<-\pi$, e.\,g. $\theta(p_1,p_2)<-\pi$. \\

In Section~\ref{sec:singana}, we assume that $T<\infty$ and there exist constants $0<c,C<\infty$ so that $h$ satisfies~\eqref{eq:h_1} and additionally
\begin{align}\label{eq:h_2}
L_t\geq c
\qquad\text{ and }\qquad
0\leq h(t)\leq C
\end{align}
for $t\in[0,T)$ and study curvature blow ups via parabolic rescaling.
We use the distance comparison principles from Sections~\ref{sec:noncompact} and~\ref{sec:noncollapsing} in the same fashion as for CSF in~\cite{Huisken95} to exclude type-II singularities and conclude that the flow exists for all positive times. \\

In Section~\ref{sec:convexity}, we assume that a solution is immortal, that is, it exists for all positive times and the global term satisfies the following.
Let $\delta\in(0,\infty)$ be given so that $\delta A_0$ is the desired limit area and
\begin{align}\label{eq:h_3a}
h(t)=(1-\gamma)\frac{2\pi}{L_t}+\gamma\frac1{2\pi}\intst\kappa^2\,ds_t\,,
\end{align}
where $\gamma=(\delta-1)A_0/\big(L_0^2/4\pi-A_0\big)$.
We prove that the curves become convex in finite time. 
The global term above ensures that the enclosed area is bounded away from zero and the length is bounded away from infinity throughout the flow. 
In Section~\ref{sec:longtimebehaviour}, we assume that an immortal solution of~\eqref{eq:ccf} with $h$ satisfying~\eqref{eq:h_3a} is convex.
We expand Gage's and Pihan's results and show smooth and exponential convergence to a round circle. \\

Note that the global term~\eqref{eq:h_ap} of the APCSF satisfies conditions~\eqref{eq:h_1},~\eqref{eq:h_2} and~\eqref{eq:h_3a}. 
The global term~\eqref{eq:h_lp} of the LPCF satisfies a priori~\eqref{eq:h_1} and~\eqref{eq:h_3a}. \\

This paper extends results from the author's PhD thesis~\cite{Dittberner18}.\\

\textbf{Acknowledgements.}
The author wants to thank Klaus Ecker, Theodora Bourni, Julian Scheuer, Mat Langford and Stephen Lynch for interest in the work and helpful discussions.

\section{Evolution equations and first consequences}\label{sec:evo}

Let $\Do\in\{\Sp\!,\R\}$ and $X:\Do\to\R^2$ be a smooth, embedded curve with length element $v:\Do\to\R$ by $v(p):=\big\Vert\fracd{}{p}X(p)\big\Vert$.
For a fixed point $p_0\in\Do$, the arc length parameter $s:\Do\to[0,L]$ is given by $s(p):=\int_{p_0}^pv(r)\,dr$, so that $ds=vdp$ and $\fracd{}{s}=\frac1{v}\fracd{}{p}$.
For $\Si=X(\Sp)$, the arc length parameter is given by $s:\Sp\to\Sp_{L/2\pi}$ and $\tilde X:=X\circ s^{-1}:\Sp_{L/2\pi}\to\R^2$ parametrises $\Si$ by arc length.
For $\Si=X(\R)$, the arc length parameter is given by $s:\R\to\R$.
The unit tangent vector field $\btau$ to $\Si$ in direction of the arc length parametrisation is given by $\btau:=\fracd{}{s}\tilde X$.
The outward unit normal is given by $\nnu:=(\btau_2,-\btau_1)$.
We define the curvature by
$$\kappa:=-\left\langle\fracd{\btau}{s},\nnu\right\rangle=\left\langle\btau,\fracd{\nnu}{s}\right\rangle$$
and the curvature vector by $\bkappa:=-\kappa\nnu$.
The Frenet--Serret equations read as
$$\fracd{\btau}{s}=-\kappa\nnu\qquad\text{ and }\qquad
\fracd{\nnu}{s}=\kappa\btau\,.$$
Let $X:\Do\!\times[0,T)\to\R^2$ be a one parameter family of maps.
For fixed $t\in[0,T)$, we can parametrise $\St=X(\Do,t)$ by arc length via the arc length parameter $s(\,\cdot\,,t)$, where $s(\Do,t)\in\{\Sp_{L_t/2\pi},\R\}$ and the arc length parametrisation is given by $\tilde X(\,\cdot\,,t)=X(\,\cdot\,,t)\circ s^{-1}(\,\cdot\,,t):s(\Do,t)\to\R^2$.
The evolution equation~\eqref{eq:ccf} applied to the arc length parametrisation reads
$$\fracp{\tilde X}{t}(s,t)=\fracp{^2\tilde X}{s^2}(s,t)+h(t)\tilde\nnu(s,t)$$
for $s\in s(\Do,t)$, where $\tilde\nnu(s,t)=\tilde\nnu(s(p,t),t)=\nnu(p,t)$ and we used the identity $\Delta_{\Si}\tilde X=\fracd2{s^2}\tilde X=\bkappa$ for the curvature vector.
Whenever we will calculate via the arc length parametrisation, we will do so at a fixed time.
Since the images $X(\Do,t)=\tilde X(s(\Do,t),t)$ are the same and $X$ and $\tilde X$ only differ by a tangential diffeomorphism, we will omit the ``$\sim$'' in the following above geometric quantities related to $\tilde X$ if these depend on $s$ rather than $p$.

\begin{Lemma}[Gage~{\cite{Gage86}}]\label{lem:evolutionequations}
Let $X:\Do\times(0,T)\to\R^2$ be a solution of~\eqref{eq:ccf}.
Then, for $t\in(0,T)$,
\begin{alignat*}{2}
\fracp{v}{t}&=\kappa(h-\kappa)v\,,\qquad
\fracp{}{t}\fracp{}{s}&&=\fracp{}{s}\fracp{}{t}-\kappa(h-\kappa)\fracp{}{s}\,, \\
\fracp{\btau}{t}&=-\fracp{\kappa}{s}\nnu\,,\qquad\qquad
\fracp{\nnu}{t}&&=\fracp{\kappa}{s}\btau\,,\qquad\qquad
\fracp{\kappa}{t}=\fracp{^2\kappa}{s^2}-(h-\kappa)\kappa^2\,.
\end{alignat*}
\end{Lemma}

\begin{Cor}[Huisken~{\cite[Thm.~1.3]{Huisken87}}]\label{cor:strongmaxpkappa}
Let $X:\Sp\!\times[0,T)\to\R^2$ be a solution of~\eqref{eq:ccf} and let $\kappa\geq0$ on $\Si_0$.
Then $\kappa>0$ on $\St$ for all $t\in(0,T)$.
\end{Cor}

\begin{Lemma}[Gage~{\cite{Gage86}}]\label{lem:dtL}
Let $X:\Sp\!\times(0,T)\to\R^2$ be a solution of~\eqref{eq:ccf}.
Then, for $t\in(0,T)$,
$$\fracd{A}{t}=hL-2\pi
\qquad\text{ and }\qquad
\fracd{L}{t}=2\pi h-\intst\kappa^2\,ds_t\,.$$
\end{Lemma}

\begin{Prop}\label{prop:T<infty}
Let $X:\Do\times[0,T)\to\R^2$ be a solution of~\eqref{eq:ccf} with initial curve $\Si_0$.
If $T<\infty$, then $\max_{p\in\Do}|\kappa(p,t)|\to\infty$ for $t\to T$.
\end{Prop}

\begin{proof}
Like in~\cite[Section~6.3]{Pihan98} (see also~\cite[Chapter~4]{Dittberner18}), we can bound the derivatives of the curvature in terms of the curvature as long as the curvature is bounded.
The proposition then follows like in~\cite[Thm.~8.1]{Huisken84}. 
\end{proof}

\section{Angles and local total curvature}\label{sec:theta2}

We want to exploit the relationship between angles of tangent vectors and local total curvatures and prove a strong maximum principle for the latter. \\

Define $\vartheta:\Do\times[0,T)\to\Sp$ to be the angle between the $x_1$-axis and the tangent vector, so that
$$\vartheta(p,t)=\begin{cases}
\arccos(\langle\e_1,\btau(p,t)\rangle)&\qquad\text{ if }\langle\e_2,\btau(p,t)\rangle\geq0 \\
2\pi-\arccos(\langle\e_1,\btau(p,t)\rangle)&\qquad\text{ if }\langle\e_2,\btau(p,t)\rangle<0\,.
\end{cases}$$
Since $\nnu=(\btau_2,-\btau_1)$, 
\begin{align}\label{eq:defvartheta}
\cos(\vartheta)=\langle\e_1,\btau\rangle=-\langle\e_2,\nnu\rangle
\quad\text{ and }\quad
\sin(\vartheta)=\langle\e_2,\btau\rangle=\langle\e_1,\nnu\rangle\,.
\end{align}
For a fixed time $t\in[0,T)$, we can define the angle $\tilde\vartheta$ via the arc length parameter by $\tilde\vartheta:s(\Do,t)\to[0,2\pi)$. 
As explained earlier, we can omit the ``$\sim$'' for simplicity.

\begin{Lemma}[see Gage--Hamilton~{\cite[Lem.~3.1.5]{GageHamilton86}}]\label{lem:dsvartheta}
Let $X:\Do\times[0,T)\to\R^2$ be a solution of~\eqref{eq:ccf} with initial curve $\Si_0$.
Then
$$\fracp{\vartheta}{s}=\kappa
\qquad\text{ and }\qquad
\fracp{\vartheta}{t}=\fracp{\kappa}{s}
\quad\text{ for }\;
t\in(0,T).$$
\end{Lemma}

Like in~\eqref{eq:intkappa}, we define the total local curvature $\theta:\Do\times\Do\times[0,T)\to\R$ by
\begin{align}\label{eq:deftheta}
\theta(p,q.t):=\int_p^q\kappa(r,t)\,ds_t\,,
\end{align}
where we integrate in direction of the parametrisation.
The total curvature $\alpha(t)$ is given by the full integral over the curvature as stated in~\eqref{eq:def_alpha}.
For $\Do=\Sp$ and $p,q\in[0,2\pi)$, we set
\begin{align*}
\theta(p,q.t)=
\begin{dcases}
\int_p^q\kappa\,ds_t&\qquad\text{ if }p\leq q\\
\int_p^{2\pi}\kappa\,ds_t+\int_0^q\kappa\,ds_t&\qquad\text{ if }q<p\,.
\end{dcases}
\end{align*}
Then 
\begin{align}\label{eq:2pithetatheta}
2\pi=\theta(p,q,t)+\theta(q,p,t)
\end{align}
for all $p,q\in\Sp$.
For $\Do=\R$ and $p<q$, we set $\theta(q,p,t)=-\theta(p,q,t)$.
By Lemma~\ref{lem:dsvartheta},
\begin{align}\label{eq:thetavartheta}
\theta(p,q)
=\int_p^q\kappa\,ds
=\int_p^q\frac1v\fracp{\vartheta}{r}\,vdr
=\vartheta(q)-\vartheta(p)+2\pi\omega(p,q)\,,
\end{align}
where $\omega\in\Z$ is the local winding number.
Hence, $\theta$ is the angle between the tangent vectors at two points on the curve modulo the local winding number.
If a curve $\Si=X(\Do)$ is embedded and convex, then 
$0\leq\theta(p,q)<\alpha$
for all $p,q\in\Do$.
For $\Do=\Sp$ and fixed $p\in\Sp$,
\begin{align}\label{eq:thetalim0}
\lim_{q\searrow p}\theta(p,q)
=\lim_{q\searrow p}\int_p^q\kappa\,vdr
=0
\end{align}
and
\begin{align}\label{eq:thetalim2pi}
\lim_{q\nearrow p}\theta(p,q)
=\lim_{q\nearrow p}\int_0^q\kappa\,vdr+\int_p^{2\pi}\kappa\,vdr
=\intsp\kappa\,vdr
=2\pi\,.
\end{align}
Hence, $\theta$ is discontinuous along the diagonal $\{p=q\}\subset\Sp\!\times\Sp$.

\begin{Lemma}\label{lem:minmaxtheta}
Let $\Si=X(\Sp)$ be an embedded, closed curve. 
Then 
$$\sup_{\Sp\!\times\Sp}\theta=2\pi-\min_{\Sp\!\times\Sp}\theta\,.$$
\end{Lemma}

\begin{proof}
Let the maximum of $\theta$ be attained at $p_0,q_0\in\Sp$, that is, by~\eqref{eq:2pithetatheta},
\begin{align}\label{eq:2pithetatheta_2}
\max_{\Sp\!\times\Sp}\theta=\theta(p_0,q_0)=2\pi-\theta(q_0,p_0)\,.
\end{align}
Then, for all $p,q\in\Sp$, $p\ne q$, by~\eqref{eq:2pithetatheta} and~\eqref{eq:2pithetatheta_2},
$$2\pi-\theta(q_0,p_0)
=\theta(p_0,q_0)
\geq\theta(q,p)
=2\pi-\theta(p,q)\,.$$
Consequently, $\theta(q_0,p_0)\leq\theta(p,q)$ for all $p,q\in\Sp$, $p\ne q$, which implies with~\eqref{eq:2pithetatheta_2},
\[\min_{\Sp\!\times\Sp}\theta=\theta(q_0,p_0)
=2\pi-\max_{\Sp\!\times\Sp}\theta\,.\qedhere\]
\end{proof}

\begin{Lemma}\label{lem:minmaxthetaentire}
Let $\Si=X(\R)$.
Then 
$$\sup_{\R\times\R}\theta\leq\alpha-2\inf_{\R\times\R}\theta\,.$$
\end{Lemma}

\begin{proof}
For $p,q\in\R$, $p<q$, 
\[\alpha
=\int_{-\infty}^p\kappa\,ds+\int_p^q\kappa\,ds+\int_q^\infty\kappa\,ds 
\geq2\inf_{\R\times\R}\theta+\theta(p,q)\,.\qedhere\]
\end{proof}

For $t\in[0,T)$, we define 
$$\theta_{\min}(t):=\min_{(p,q)\in\Sp\!\times\Sp}\theta(p,q,t)\leq0
\quad\text{ and }\quad
\theta_{\inf}(t):=\inf_{(p,q)\in\R\times\R}\theta(p,q,t)\leq\alpha\,.$$

\begin{Thm}\label{thm:dtdstheta}
Let $X:\Do\times(0,T)\to\R^2$ be a solution of~\eqref{eq:ccf}.
Then
$$\left(\fracp{}{t}-\Delta_{\St}\right)\theta(p,q,t)=0$$
for all $p,q\in\Do$ ($p\ne q$ for $\Do=\Sp$) and $t\in(0,T)$.
Moreover, let $t_0\in(0,T)$.
\begin{enumerate}[(i)]
\item For $\Do=\Sp$, suppose $\theta_{\min}(t_0)<0$, then $\theta_{\min}(t_0)<\theta_{\min}(t)$ for all $t\in(t_0,T)$.
\item For $\Do=\R$, let~\eqref{eq:alpha} be satisfied. 
Suppose $\theta_{\inf}(t_0)<\min\{0,\alpha\}$, then $\theta_{\inf}(t_0)<\theta_{\inf}(t)$ for all $t\in(t_0,T)$. 
\end{enumerate}
\end{Thm}

\begin{proof}
We differentiate $\theta$ at $p,q\in\Do$ ($p\ne q$ for $\Do=\Sp$) in direction of $\btau$ with $\btau(\theta) = \fracp{}{s}\theta$, and use Lemma~\ref{lem:dsvartheta} to obtain
$$\btau_p(\theta)
=-\kappa_p
\qquad\text{ and }\qquad
\btau_q(\theta)
=\kappa_q\,,$$
as well as
\begin{align}\label{eq:ds2theta}
\Delta_{\St}\theta
=\btau_p^2(\theta)+\btau_q^2(\theta)
=\btau_q(\kappa_q)-\btau_p(\kappa_p)
=\fracp{\vartheta_q}{t}-\fracp{\vartheta_p}{t}
=\fracp{\theta}{t}\,.
\end{align}
Since $X$ is smooth, $\theta$ is smooth in $\Do\times\Do\times[0,T)$.

(i) Let $\Do=\Sp$. 
The set $S:=\Sp\!\times\Sp\setminus\{p=q\}$ is an oriented cylinder.
By~\eqref{eq:thetalim0} and~\eqref{eq:thetalim2pi}, the closure $\bar S$ has two boundaries
$$(\pa S)_-=\left\{(p,p)\,\big|\, p\in\Sp\right\}\quad\text{ and }\quad
(\pa S)_+=\left\{\left(\lim_{r\nearrow p}r,p\right)\,\bigg|\, p\in\Sp\right\}\,,$$
where $\theta\equiv0$ on $(\pa S)_-\times[0,T)$ and $\theta\equiv2\pi$ on $(\pa S)_+\times[0,T)$.
The claim now follows from the strong maximum principle with boundary conditions.

(ii) Let $\Do=\R$.
By~\eqref{eq:alpha}, $\Si_t$ is up to translation smoothly asymptotic to two time-independent lines, that is, $|\btau_p\kappa(p,t)|\to0$ uniformly for $p\to\pm\infty$, by~\eqref{eq:ds2theta}, $\lim_{p\to\pm\infty}\vartheta_p$ is constant in time.
For $p\in\R$, define
$$\theta^-(p,t):=\lim_{q\to-\infty}\theta(q,p,t)
\quad\text{ and }\quad
\theta^+(p,t):=\lim_{q\to\infty}\theta(p,q,t)\,.$$
If $\theta_{\inf}<\min\{0,\alpha\}$, $\theta_{\inf}$ either equals $\theta_{\inf}^+$ or $\theta_{\inf}^-$, or is attained at a local minimum.
Suppose $\theta_{\inf}<0$ is attained at a local minimum.
We have that $\theta\equiv0$ on $\{p=q\}\times[0,T)=\pa\{p\ne q\}\times[0,T)$.
The strong maximum principle yields that $\theta$ is strictly decreasing.
Suppose $\theta_{\inf}(t)<\min\{0,\alpha\}$ equals $\theta_{\inf}^+$ or $\theta_{\inf}^-$.
By~\eqref{eq:ds2theta},
$$\left(\fracp{}{t}-\Delta_{\St}\right)\theta^{\pm}=0\,.$$
If $\theta_{\inf}^{\pm}<\min\{0,\alpha\}$, then $\theta_{\inf}^{\pm}$ is attained at a point $p\in\R$.
The strong maximum principle yields that $\theta_{\inf}^{\pm}$ is strictly increasing as long as $\theta_{\inf}^{\pm}<\min\{0,\alpha\}$.
\end{proof}

We define the extrinsic distance $d:\Do\times\Do\times[0,T)\to\R$ by
$$d(p,q,t):=\Vert X(q,t)-X(p,t)\Vert_{\R^2}$$
and the vector $\bw:\big(\Do\times\Do\times[0,T)\big)\setminus\{d=0\}\to\R^2$ by
$$\bw(p,q,t):=\frac{X(q,t)-X(p,t)}{d(p,q,t)}\,.$$

\begin{Lemma}\label{lem:thetabeta}
Let $\Si=X(\Do)$ be an embedded curve and $p,q\in\Do$ with $d(p,q)\ne0$.
Let $\langle\bw,\btau_p\rangle=\langle\bw,\btau_q\rangle=\cos(\beta/2)$ for $\beta\in[0,\pi]$.
Then either
\begin{enumerate}[(i)]
\item $\langle\bw,\nnu_p\rangle=-\langle\bw,\nnu_q\rangle=-\sin(\beta/2)$ and $\theta(p,q)=2\pi k+\beta$,
\item $\langle\bw,\nnu_p\rangle=-\langle\bw,\nnu_q\rangle=\sin(\beta/2)$ and $\theta(p,q)=2\pi k-\beta$, or
\item $\langle\bw,\nnu_p\rangle=\langle\bw,\nnu_q\rangle=\pm\sin(\beta/2)$ and $\theta(p,q)=2\pi k$
\end{enumerate}
for $k\in\Z$.
\end{Lemma}

\begin{proof}
The angles are invariant under rotations in the plane, thus we may assume $\bw=\e_1$.
Since $\beta\in[0,\pi]$, the definition~\eqref{eq:defvartheta} of $\vartheta\in[0,2\pi)$ yields
$$\cos(\vartheta_p)=\langle\e_1,\btau_p\rangle=\cos\!\left(\frac\beta2\right)\geq0
\quad\text{ and }\quad
\sin(\vartheta_p)=\langle\e_1,\nnu_p\rangle=\pm\sin\!\left(\frac\beta2\right)$$
as well as
$$\cos(\vartheta_q)=\langle\e_1,\btau_q\rangle=\cos\!\left(\frac\beta2\right)\geq0
\quad\text{ and }\quad
\sin(\vartheta_q)=\langle\e_1,\nnu_q\rangle=\pm\sin\!\left(\frac\beta2\right)\,.$$
Hence, by~\eqref{eq:thetavartheta},
\begin{align}\label{eq:theta_2}
\vartheta_p,\vartheta_q\in\left\{\frac\beta2,2\pi-\frac\beta2\right\}\quad\text{ and }\quad
\theta=2\pi\omega+\vartheta_q-\vartheta_p\,,
\end{align}
where $\omega\in\Z$. 

(i) Assume that $\sin(\vartheta_p)=-\sin(\beta/2)\leq0$ and $\sin(\vartheta_q)=\sin(\beta/2)\geq0$.
From~\eqref{eq:theta_2} it follows that $\vartheta_p=2\pi-\beta/2$, $\vartheta_q=\beta/2$ and $\theta=2\pi(\omega+1)+\beta$.

(ii) Assume that $\sin(\vartheta_p)=\sin(\beta/2)\geq0$ and $\sin(\vartheta_q)=-\sin(\beta/2)\leq0$.
From~\eqref{eq:theta_2} it follows that $\vartheta_p=\beta/2$, $\vartheta_q=2\pi-\beta/2$ and $\theta=2\pi(\omega+1)-\beta$.

(iii) Assume that $\sin(\vartheta_p)=\pm\sin(\beta/2)$ and $\sin(\vartheta_q)=\pm\sin(\beta/2)$.
From~\eqref{eq:theta_2} it follows that either $\vartheta_p=\vartheta_q=\beta/2$ or $\vartheta_p=\vartheta_q=2\pi-\beta/2$ and thus $\theta=2\pi\omega$.
\end{proof}

\section{Distance comparison principle for non\-compact curves}\label{sec:noncompact}

We adapt the methods from Huisken~\cite{Huisken95} to obtain estimates that imply a certain non-collapsing behaviour of the evolving curves.\\

The intrinsic distance $l:\Do\times\Do\times[0,T)\to\R$ is given by
$$l(p,q,t):=\int_p^qv(r,t)\,dr\,.$$
We set $d/l\equiv1$ on $\{p=q\}\times[0,T)$, then $d/l\in C^0(\R\times\R\times[0,T))$.
Embedded curves satisfy $(d/l)(p,q)>0$ for all $p,q\in\R$.
If a curve is not a line, then there exist $p,q\in\R$ so that $d(p,q)<l(p,q)$ and thus $\inf_{\R\times\R}(d/l)<1$. 

\begin{Lemma}
\label{lem:mindltheta2pik}
Let $\Si=X(\R)$ be an embedded curve.
Let $p,q\in\R$, $p\ne q$, such that $\Si$ crosses the connecting line between $X(p)$ and $X(q)$ at $X(r)$ with $r\notin[p,q]$.
Then $(d/l)(p,q)$ cannot be the infimum.
\end{Lemma}

\begin{proof}
Let $X(\R)$ cross the connecting line between $X(p)$ and $X(q)$ at $X(r)$.
Then $X(r)=X(p)+\bw(p,q)\Vert X(r)-X(p)\Vert$.
Set $d:=d(p,q)$, $d_1:=d(p,r)$ and $d_2:=d(r,q)$.
Then $d=d_1+d_2$.
Furthermore, set $l:=l(p,q)$, $l_1:=l(p,r)$ and $l_2:=l(r,q)$.
If $p<q<r$, then $l<l_1$ and $d/l>d_1/l_1$.
If $r<p<q$, then $l<l_2$ and $d/l>d_2/l_2$.
Thus, $d/l$ cannot be a global minimum.
\end{proof}

Now we can prove a similar result to~\cite[Thm.~2.1]{Huisken95}.

\begin{Thm}
\label{thm:minimumdl}
Let $\Si_0=X_0(\R)$ be a smooth, embedded curve satisfying~\eqref{eq:alpha} and~\eqref{eq:intkappageqminuspi}. 
Let $X:\R\times[0,T)\to\R^2$ be a solution of~\eqref{eq:ccf} satisfying~\eqref{eq:h_1}, and with initial curve $\Si_0$.
Then there exists a constant $c(\Si_0)>0$ such that
$$\inf_{(p,q,t)\in\R\times\R\times[0,T)}\frac dl(p,q,t)\geq c\,.$$
\end{Thm}

\begin{proof}
Let $\Si_0=X_0(\R)$ be an embedded curve satisfying~\eqref{eq:alpha} and~\eqref{eq:intkappageqminuspi}.
Then
\begin{align}\label{eq:limdl}
\inf_{t\in[0,T)}\lim_{p\to\infty}\frac dl(p,-p,t)=:c_1(\Si_0)\in(0,1]
\end{align}
for all $t\in(0,T)$.
Lemma~\ref{lem:minmaxthetaentire} implies that $\theta_0\in[-\pi,\alpha+2\pi]$.
From the maximum principle for $\theta$, Theorem~\ref{thm:dtdstheta}, it follows that
\begin{align}\label{eq:thetainteral2dl}
\theta(p,q,t)\in(-\pi,\alpha+2\pi)
\end{align}
for all $p,q\in\Sp$ and $t\in(0,T)$.
Since $d/l$ is continuous and initially positive, there exists a time $T'\in(0,T]$ so that $d/l>0$ on $[0,T')$.
Fix $t_0\in(0,T')$.
If $\Si_{t_0}$ is a line, then $d/l\equiv1$ on $\R\times\R$. 
Assume that $\Si_{t_0}$ is not a line so that $\inf_{\R\times\R}(d/l)<1$ at $t_0$.
Let $p,q\in\R$, $p\ne q$, be points where a local spatial minimum of $d/l$ at $t_0$ is attained and assume w.l.o.g.\ that $s(p,t_0)<s(q,t_0)$.
We have for all $\bxi\in T_{X(p,t_0)}\Si_{t_0}\bigoplus T_{X(q,t_0)}\Si_{t_0}$,
$$0<\frac dl(p,q,t_0)<1\,,\quad
\bxi\!\left(\frac dl\right)\!(p,q,t_0)=0
\quad\text{ and }\quad
\bxi^2\!\left(\frac dl\right)\!(p,q,t_0)\geq0\,.$$
In the following, we always calculate at the point $(p,q,t_0)$.
The spatial derivatives of $d$ and $l$ are all given in~\cite{Huisken95} (for detailed calculations, see~\cite[Lems.~6.2 and~7.4]{Dittberner18}).
The first spatial derivative of $d/l$ at $(p,t_0)$ in direction of the vector $\bxi=\btau_p\oplus0$ is given by
\begin{align}\label{eq:wtaupdl}
0=(\btau_p\oplus0)\!\left(\frac dl\right)
=-\frac1l\langle\bw,\btau_p\rangle+\frac d{l^2}\,.
\end{align}
At $(q,t_0)$ and for the vector $\bxi=0\oplus\btau_q$, we have
\begin{align}\label{eq:wtauqdl}
0=(0\oplus\btau_q)\!\left(\frac dl\right) 
=\frac1l\langle\bw,\btau_q\rangle-\frac d{l^2}\,.
\end{align}
Since $d/l\in(0,1)$, and by~\eqref{eq:wtaupdl} and~\eqref{eq:wtauqdl}, there exists $\beta\in(0,\pi)$ with
\begin{align}\label{eq:wtaucosthetadl}
\langle\bw,\btau_p\rangle=\langle\bw,\btau_q\rangle
=\frac dl
=\cos\!\left(\frac\beta2\right)
\in(0,1)\,.
\end{align}
By Lemma~\ref{lem:thetabeta},~\eqref{eq:thetainteral2dl} and~\eqref{eq:wtaucosthetadl}, either
\begin{align*}
\langle\bw,\nnu_p\rangle=-\langle\bw,\nnu_q\rangle=-\sin\!\left(\frac\beta2\right)&\quad\text{ and }\quad
2\pi k+\beta=\theta\in(0,\pi)\cup(2\pi,\alpha+2\pi)\,, \\
\langle\bw,\nnu_p\rangle=-\langle\bw,\nnu_q\rangle=\sin\!\left(\frac\beta2\right)&\quad\text{ and }\quad
2\pi k-\beta=\theta\in(-\pi,0)\cup(\pi,2\pi)\,,\quad\text{ or } \\
\langle\bw,\nnu_p\rangle=\langle\bw,\nnu_q\rangle=\pm\sin\!\left(\frac\beta2\right)&\quad\text{ and }\quad
\theta\in\{0,2\pi\}
\end{align*}
for $k\in\Z$.
We use the evolution equation~\eqref{eq:ccf} and Lemma~\ref{lem:evolutionequations} to differentiate the ratio in time,
\begin{align}\label{eq:dtdl}
\fracp{}{t}\left(\frac dl\right)
&=\frac1l\big(h\langle\bw,\nnu_q-\nnu_p \rangle
		+\langle\bw,\bkappa_q-\bkappa_p\rangle\big) \notag\\
	&\quad+\frac d{l^2}\left(\int_p^q\kappa^2\,ds_t-h\int_p^q\kappa\,ds_t\right)\,.
\end{align}
We are now considering four different cases.\\

(i) Assume that
\begin{align}\label{eq:wnusinthetadl} 
\langle\bw,\nnu_p\rangle=-\langle\bw,\nnu_q\rangle=-\sin\!\left(\frac\beta2\right)\quad\text{ and }\quad
\beta=\theta\in(0,\pi)\,.
\end{align}
Adding~\eqref{eq:wtaupdl} to~\eqref{eq:wtauqdl} yields $\langle\bw,\btau_q-\btau_p\rangle=0$.
For unit tangent vectors, we have
$$\langle\btau_p+\btau_q,\btau_q-\btau_p\rangle=\Vert\btau_p\Vert^2-\Vert\btau_q\Vert^2=0\,.$$
Thus, $\bw$ and $\btau_p+\btau_q$ are both perpendicular to $\btau_q-\btau_p$ and are therefore parallel, that is, $\measuredangle(\bw,\btau_q+\btau_p)=0$.
Using $\Vert\bw\Vert=1$, we calculate
\begin{align}\label{eq:wtautaupplustauqdl}
0\leq2\frac dl
=\langle\bw,\btau_q+\btau_p\rangle 
=\Vert\btau_q+\btau_p\Vert\,.
\end{align}
By~\eqref{eq:wnusinthetadl},
\begin{align}\label{eq:wkappasinthetadl}
\langle\bw,\bkappa_q-\bkappa_p\rangle
=-\kappa_q\langle\bw,\nnu_q\rangle+\kappa_p\langle\bw,\nnu_p\rangle
=-(\kappa_p+\kappa_q)\sin\!\left(\frac\theta2\right)\,.
\end{align} 
We differentiate $d/l$ at $(p,q,t_0)$ twice with respect to the vector $\bxi=\btau_p\ominus\btau_q$ and calculate with~\eqref{eq:wtautaupplustauqdl} and~\eqref{eq:wkappasinthetadl},
\begin{align*}
0&\leq(\btau_p\ominus \btau_q)^2\left(\frac dl\right) 
=\frac1l\langle\bw,\bkappa_q-\bkappa_p\rangle 
=-\frac1l(\kappa_p+\kappa_q)\sin\!\left(\frac\theta2\right)\,.
\end{align*}
We abbreviate $\kappa:=(\kappa_p+\kappa_q)/2$ and obtain
\begin{align*}
2\kappa\sin\!\left(\frac\theta2\right)\leq0\,.
\end{align*}
Since $\sin(\theta/2)>0$ for $\theta\in(0,\pi)$, we conclude $\kappa\leq0$ and
$$h-\kappa\geq h>h-\frac\theta l\,.$$
Furthermore, the inequality 
$$\sin\!\left(\frac\theta2\right)>\frac\theta2\cos\!\left(\frac\theta2\right)>0$$
holds for $\theta\in(0,\pi)$. 
Hence, 
\begin{align}\label{eq:proofdtdlcase2a2}
(h-\kappa)\sin\!\left(\frac\theta2\right)>\frac\theta2\cos\!\left(\frac\theta2\right)\left(h-\frac\theta l\right)
\end{align}
for $\theta\in(0,\pi)$.
Cauchy--Schwarz and the definition~\eqref{eq:deftheta} of $\theta$ imply
\begin{align}\label{eq:CSthetadl}
\int_p^q\kappa^2\,ds_t
\geq\frac1l\left(\int_p^q\kappa\,ds_t\right)^{2}
=\frac{\theta^2}l\,.
\end{align}
Then~\eqref{eq:dtdl},~\eqref{eq:wtaucosthetadl},~\eqref{eq:wnusinthetadl},~\eqref{eq:wkappasinthetadl},~\eqref{eq:proofdtdlcase2a2} and~\eqref{eq:CSthetadl} yield
\begin{align*}
\fracp{}{t}\left(\frac dl\right)
&\geq\frac2l\left((h-\kappa)\sin\!\left(\frac\theta2\right)
		-\frac\theta2\cos\!\left(\frac\theta2\right)\left(h-\frac\theta l\right)\right)
>0
\end{align*}
at $(p,q,t_0)$. \\

(ii) Assume that $\langle\bw,\nnu_p\rangle=\langle\bw,\nnu_q\rangle=\pm\sin(\beta/2)$ and $\theta=0$.
By~\eqref{eq:wtaucosthetadl},
\begin{align}\label{eq:wnusinthetadl_0} 
\btau_p=\btau_q\qquad\text{ and }\qquad\nnu_p=\nnu_q\,.
\end{align}
We differentiate $d/l$ at $(p,q,t_0)$ twice with respect to the vector $\bxi=\btau_p\oplus\btau_q$ and calculate with~\eqref{eq:wnusinthetadl_0},
\begin{align}\label{eq:wtautaupplustauqdl_0}
0&\leq(\btau_p\oplus\btau_q)^2\left(\frac dl\right) 
=\frac1l\langle\bw,\bkappa_q-\bkappa_p\rangle\,.
\end{align}
We conclude with~\eqref{eq:deftheta},~\eqref{eq:dtdl},~\eqref{eq:wnusinthetadl_0},~\eqref{eq:wtautaupplustauqdl_0},
\begin{align*}
\fracp{}{t}\left(\frac dl\right)
\geq\frac d{l^2}\int_p^q\kappa^2\,ds_t
>0
\end{align*}
at $(p,q,t_0)$. \\

(iii) Assume that $\langle\bw,\nnu_p\rangle=-\langle\bw,\nnu_q\rangle=\sin(\beta/2)$ and $-\beta=\theta\in(-\pi,0)$. 
Again by continuity of $d/l$, there exists $t_1(\Si_0)\in(0,T')$ so that
\begin{align}\label{eq:Cstar1dl}
\inf_{\R\times\R}\frac dl(\,\cdot\,,\,\cdot\,,t)\geq\frac12\inf_{\R\times\R}\frac dl(\,\cdot\,,\,\cdot\,,0)
\end{align}
for all $t\in[0,t_1]$.
If $t_0\in(t_1,T)$, Theorem~\ref{thm:dtdstheta} and~\eqref{eq:thetainteral2dl} yield $-\pi<\theta_{\inf}(t_1)\leq\theta_{\inf}(t_0)<0$ so that~\eqref{eq:wtaucosthetadl} and the monotone behaviour of the cosine on $(-\pi,0)$ imply
\begin{align}\label{eq:Cstar2dl}
\frac dl=\cos\!\left(\frac\theta2\right)
\geq\cos\!\left(\frac{\theta_{\inf}(t_0)}2\right)
\geq\cos\!\left(\frac{\theta_{\inf}(t_1)}2\right)>0\,.
\end{align}
We deduce with~\eqref{eq:Cstar1dl} and~\eqref{eq:Cstar2dl} that
\begin{align}\label{eq:c2}
\frac dl\geq
\min\left\{\frac12\inf_{\R\times\R}\frac dl(\,\cdot\,,\,\cdot\,,0),\cos\!\left(\frac{\theta_{\inf}(t_1)}2\right)\right\}=:c_2(\Si_0)>0
\end{align}
at $(p,q,t_0)$.\\

(iv) Assume that $\theta\in(\pi,\alpha+2\pi)$.
By~\eqref{eq:wtaucosthetadl}, $\langle\bw,\btau_q\rangle=\langle\bw,\btau_p\rangle=d/l\in(0,1)$.
Since $\Si_{t_0}$ is embedded with ends going to $\infty$ and $X(\,\cdot\,,t_0)$ is continuous, the curve has to cross the line segment between $X(p,t_0)$ and $X(q,t_0)$ at least once at $X(r,t_0)$ with $r\notin[p,q]$.
Lemma~\ref{lem:mindltheta2pik} implies that $d/l$ cannot attain the infimum at $(p,q,t_0)$ (it could still, however, attain a local minimum a this point).
Hence,
$$\frac dl>\inf_{\R\times\R}\frac dl(\,\cdot\,,\,\cdot\,,t_0)\,,$$
where $\inf_{\R\times\R}d/l(\,\cdot\,,\,\cdot\,,t_0)$ is either the infimum from~\eqref{eq:limdl} or a local minimum as discussed in cases~(i),~(ii) and~(iii). \\

Assume that $d/l$ falls below $c:=\min\{c_1,c_2\}$, where $c_1$ and $c_2$ are given in~\eqref{eq:limdl} and~\eqref{eq:c2}, and attains $\Lambda\in(0,c)$ for the first time at time $t_2\in(0,T)$ and points $p,q\in\Sp$, $p\ne q$, so that
\begin{align}\label{eq:c1mindl}
c>\Lambda=\frac dl(p,q,t_2)=\inf_{\R\times\R}\frac dl(\,\cdot\,,\,\cdot\,,t_2)
\end{align}
is the infimum and
\begin{align}\label{eq:dtdlleq0}
\fracp{}{t}_{|_{t=t_2}}\!\left(\frac dl\right)(p,q,t)\leq0\,.
\end{align}
Cases~(i) and~(ii) contradict~\eqref{eq:dtdlleq0}, and cases~(iii) and~(iv) contradict~\eqref{eq:c1mindl}. 
\end{proof}

\begin{Cor}
\label{cor:embeddednessdl}
Let $\Si_0=X_0(\R)$ be a smooth, embedded curve satisfying~\eqref{eq:alpha} and~\eqref{eq:intkappageqminuspi}.
Let $X:\R\times[0,T)\to\R^2$ be a solution of~\eqref{eq:ccf} satisfying~\eqref{eq:h_1}, and with initial curve $\Si_0$.
Then $\St=X(\R,t)$ is embedded for all $t\in(0,T)$.
\end{Cor}

\begin{Rem}
Counterexample~\ref{ex:cexample} shows that in order for embeddedness to be preserved it is crucial to assume that the initial local total curvature lies above $-\pi$. 
\end{Rem}

\section{Distance comparison principle for closed curves}\label{sec:noncollapsing}

We continue to adapt the methods from Huisken~\cite{Huisken95}.
Let $X(\Sp)$ be a circle of radius $R$.
Then 
$$\frac{d(p,q)}2
=\frac{L}{2\pi}\sin\!\left(\frac{\pi l(p,q)}{L}\right)$$
for all $p,q\in\Sp$.
This motivates the definition of the function $\psi:\Sp\!\times\Sp\!\times[0,T)\to\R$ with
\begin{align}\label{eq:defpsi}
\psi(p,q,t):=\frac{L_t}\pi\sin\!\left(\frac{\pi l(p,q,t)}{L_t}\right)\,,
\end{align}
where $L_t<\infty$.
We set $d/\psi\equiv1$ on $\{p=q\}\times[0,T)$, then $d/\psi\in C^0(\Sp\!\times\Sp\!\times[0,T))$.

\begin{Rem}\label{rem:dpsi}
Since $\sin(\pi-\alpha)=\sin(\alpha)$,
we have $\psi(p,q,t)=\psi(q,p,t)$.
Hence, we will later assume that $l\leq L/2$.
Embedded curves satisfy $d/\psi>0$.
If a closed curve $\St$ is not a circle, then there exist $p,q\in\Sp$ so that $d(p,q,t)<\psi(p,q,t)$ and thus $\min_{\Sp\!\times\Sp}(d/\psi)<1$. 
\end{Rem}

\begin{Lemma}
\label{lem:mindpsitheta2pik}
Let $\Si=X(\Sp)$ be an embedded, closed curve.
Let $p,q\in\Sp$, $p\ne q$, such that $\Si$ crosses the connecting line between $X(p)$ and $X(q)$.
Then $(d/\psi)(p,q)$ cannot be a global minimum.
\end{Lemma}

\begin{proof}
Let $\Si=X(\Sp)$ be an embedded, closed curve that crosses the connecting line between $X(p)$ and $X(q)$.
That is, there exists an $r\in\Sp$, $r\ne p,q$, with $X(r)=X(p)+\bw(p,q)\Vert X(r)-X(p)\Vert$.
Set $d:=d(p,q)$, $d_1:=d(p,r)$ and $d_2:=d(r,q)$.
Then
\begin{align}\label{eq:d1plusd2}
d=d_1+d_2\,.
\end{align}
Furthermore, set $l:=l(p,q)$, $l_1:=l(p,r)$ and $l_2:=l(r,q)$
and $\psi:=\psi(p,q)$, $\psi_1:=\psi(p,r)$ and $\psi_2:=\psi(r,q)$ and assume that $d/\psi$ attains its global minimum at $(p,q)$.
We parametrise $\Si$ by arc length, so that $s(p)=0$.
Then we have either $0=s(p)<s(r)<s(q)$ with $l=l_1+l_2$ and 
$$\frac\pi{L}\psi
=\sin\!\left(\frac{\pi l}{L}\right)
<\sin\!\left(\frac{\pi l_1}{L}\right)+\sin\!\left(\frac{\pi l_2}{L}\right)
=\frac\pi{L}\psi_1+\frac\pi{L}\psi_2\,,$$
or we have $0=s(p)<s(q)<s(r)$ with $l=L-(l_1+l_2)$ and likewise
\begin{align*}
\frac\pi{L}\psi
=\sin\!\left(\frac{\pi l}{L}\right)
=\sin\!\left(\frac{\pi L}{L}-\frac{\pi(l_1+l_2)}{L}\right) 
<\sin\!\left(\frac{\pi l_1}{L}\right)+\sin\!\left(\frac{\pi l_2}{L}\right) 
=\frac\pi{L}\psi_1+\frac\pi{L}\psi_2\,.
\end{align*} 
Since $d/\psi$ is a global minimum, we can estimate with~\eqref{eq:d1plusd2} and the above,
$$\frac{d_1}{\psi_1}\geq\frac d\psi>\frac{d_1+d_2}{\psi_1+\psi_2}
\qquad\text{ and }\qquad
\frac{d_2}{\psi_2}\geq\frac d\psi>\frac{d_1+d_2}{\psi_1+\psi_2}$$
so that 
$$d_1(\psi_1+\psi_2)>(d_1+d_2)\psi_1
\qquad\text{ and }\qquad
d_2(\psi_1+\psi_2)>(d_1+d_2)\psi_2\,.$$
Adding both inequalities yields a contradiction.
Thus, $d/\psi$ cannot be a global minimum.
\end{proof}

Now we can prove a similar result to~\cite[Thm.~2.3]{Huisken95}.

\begin{Thm}
\label{thm:minimumdpsi}
Let $\Si_0=X_0(\Sp)$ be a smooth, embedded curve satisfying~\eqref{eq:intkappageqminuspi}. 
Let $X:\Sp\!\times[0,T)\to\R^2$ be a solution of~\eqref{eq:ccf} satisfying~\eqref{eq:h_1} and with initial curve $\Si_0$.
Then there exists a constant $c(\Si_0)>0$ such that
$$\inf_{(p,q,t)\in\Sp\!\times\Sp\!\times[0,T)}\frac d\psi(p,q,t)\geq c\,.$$
\end{Thm}

\begin{proof}
Let $\Si_0=X_0(\Sp)$ be an embedded closed curve satisfying~\eqref{eq:intkappageqminuspi}.
Lemma~\ref{lem:minmaxtheta} implies that $\theta_0\in[-\pi,3\pi]$.
From the maximum principle for $\theta$, Theorem~\ref{thm:dtdstheta}, it follows that
\begin{align}\label{eq:thetainteral2}
\theta(p,q,t)\in(-\pi,3\pi)
\end{align}
for all $p,q\in\Sp$ and $t\in(0,T)$. 
Since $d/\psi$ is continuous and initially positive, there exists a time $T'\in(0,T]$ so that $d/\psi>0$ on $[0,T')$.
Fix $t_0\in(0,T')$.
If $\Si_{t_0}$ is a circle, then Remark~\ref{rem:dpsi} yields that $d/\psi\equiv1$ on $\Sp\!\times\Sp$.
Assume that $\Si_{t_0}$ is not a circle so that $\min_{\Sp\!\times\Sp}(d/\psi)<1$ at $t_0$.
Let $p,q\in\Sp$, $p\ne q$, be points where a local spatial minimum of $d/\psi$ at $t_0$ is attained and assume w.l.o.g.\ that $s(p,t_0)<s(q,t_0)$.
Again by Remark~\ref{rem:dpsi}, we can assume that $l(p,q,t_0)\leq L_{t_0}/2$.
We have for all $\bxi\in T_{X(p,t_0)}\Si_{t_0}\bigoplus T_{X(q,t_0)}\Si_{t_0}$,
$$0<\frac d\psi(p,q,t_0)<1\,,\quad
\bxi\!\left(\frac d\psi\right)\!(p,q,t_0)=0
\quad\text{ and }\quad
\bxi^2\!\left(\frac d\psi\right)\!(p,q,t_0)\geq0\,.$$
In the following, we always calculate at the point $(p,q,t_0)$.
The spatial derivatives of $d$ and $\psi$ are all given in~\cite{Huisken95} (for detailed calculations, see~\cite[Cor.~7.12 and Thm.~7.21]{Dittberner18}).
The first spatial derivative of $d/\psi$ at $(p,t_0)$ in direction of the vector $\bxi=\btau_p\oplus0$ is given by
\begin{align}\label{eq:wtaup}
0=(\btau_p\oplus0)\!\left(\frac d\psi\right)
=-\frac1\psi\langle\bw,\btau_p\rangle+\frac d{\psi^2}\cos\!\left(\frac{\pi l}{L}\right)\,.
\end{align}
At $(q,t_0)$ and for the vector $\bxi=0\oplus\btau_q$, we have
\begin{align}\label{eq:wtauq}
0=(0\oplus\btau_q)\!\left(\frac d\psi\right) 
=\frac1\psi\langle\bw,\btau_q\rangle-\frac d{\psi^2}\cos\!\left(\frac{\pi l}{L}\right)\,.
\end{align}
Since $l\in(0,L/2]$ and $d/\psi\in(0,1)$, and by~\eqref{eq:wtaup} and~\eqref{eq:wtauq}, there exists $\beta\in(0,\pi]$ with
\begin{align}\label{eq:wtaucostheta}
\langle\bw,\btau_p\rangle=\langle\bw,\btau_q\rangle
=\frac d\psi\cos\!\left(\frac{\pi l}{L}\right)
=\cos\!\left(\frac\beta2\right)
\in[0,1)\,.
\end{align}
By Lemma~\ref{lem:thetabeta} and~\eqref{eq:thetainteral2}, either
\begin{align*}
\langle\bw,\nnu_p\rangle=-\langle\bw,\nnu_q\rangle=-\sin\!\left(\frac\beta2\right)&\quad\text{ and }\quad
2\pi k+\beta=\theta\in(0,\pi]\cup(2\pi,3\pi)\,, \\
\langle\bw,\nnu_p\rangle=-\langle\bw,\nnu_q\rangle=\sin\!\left(\frac\beta2\right)&\quad\text{ and }\quad
2\pi k-\beta=\theta\in(-\pi,0)\cup[\pi,2\pi)\,,\quad\text{ or } \\
\langle\bw,\nnu_p\rangle=\langle\bw,\nnu_q\rangle=\pm\sin\!\left(\frac\beta2\right)&\quad\text{ and }\quad
\theta\in\{0,2\pi\}
\end{align*}
for $k\in\Z$.
We are now considering three different cases.\\

(i) Assume that
\begin{align}\label{eq:wnusintheta} 
\langle\bw,\nnu_p\rangle=-\langle\bw,\nnu_q\rangle=-\sin\!\left(\frac\beta2\right)\quad\text{ and }\quad
\beta=\theta\in(0,\pi]\,.
\end{align}
By~\eqref{eq:wtaucostheta}, also
$$\langle\bw,\nnu_p\rangle=-\langle\bw,\nnu_q\rangle=-\frac d\psi\sin\!\left(\frac{\pi l}{L}\right)$$
so that
\begin{align}\label{eq:sinthetasinpilL}
2\sin\!\left(\frac\theta2\right)
=\langle\bw,\nnu_q-\nnu_p\rangle
=2\frac d\psi\sin\!\left(\frac{\pi l}{L}\right)\,.
\end{align}
Like in the proof of Theorem~\ref{thm:minimumdl},
$\measuredangle(\bw,\btau_q+\btau_p)=0$.
Using $\Vert\bw\Vert=1$ and~\eqref{eq:wtaucostheta}, we calculate
\begin{align}\label{eq:wtautaupplustauq}
0\leq2\frac d\psi\cos\!\left(\frac{\pi l}{L}\right)
=\langle\bw,\btau_q+\btau_p\rangle 
=\Vert\btau_q+\btau_p\Vert\,.
\end{align}
Since $d/\psi<1$, and again by~\eqref{eq:wtaucostheta},
\begin{align}\label{eq:costhetacosdpsi}
\cos\!\left(\frac\theta2\right)
=\frac d\psi\cos\!\left(\frac{\pi l}{L}\right)
<\cos\!\left(\frac{\pi l}{L}\right)\,.
\end{align}
As $\pi l/\!L\in(0,\pi/2]$ and the cosine function is axially symmetric and monotonically decreasing on $(0,\pi/2]$,~\eqref{eq:costhetacosdpsi} implies
\begin{align}\label{eq:thetalR}
\theta>\frac{2\pi l}{L}\,.
\end{align}
By~\eqref{eq:wnusintheta},
\begin{align}\label{eq:wkappasintheta}
\langle\bw,\bkappa_q-\bkappa_p\rangle
=-\kappa_q\langle\bw,\nnu_q\rangle+\kappa_p\langle\bw,\nnu_p\rangle
=-(\kappa_p+\kappa_q)\sin\!\left(\frac\theta2\right)\,.
\end{align} 
We differentiate $d/\psi$ at $(p,q,t_0)$ twice with respect to the vector $\bxi=\btau_p\ominus\btau_q$ and calculate, using the definition~\eqref{eq:defpsi} of $\psi$,~\eqref{eq:wtautaupplustauq} and~\eqref{eq:wkappasintheta},
\begin{align*}
0&\leq (\btau_p\ominus\btau_q)^2\!\left(\frac d\psi\right) 
=-\frac1\psi(\kappa_p+\kappa_q)\sin\!\left(\frac\theta2\right)
		+\frac{4\pi^2d}{L^2\psi}\,.
\end{align*}
We abbreviate $\kappa:=(\kappa_p+\kappa_q)/2$ and obtain
\begin{align}\label{eq:kappasinthetadpsi}
2\kappa\sin\!\left(\frac\theta2\right)\leq\frac{4\pi^2}{L^2}d\,.
\end{align}
Since the sine function is positive and monotonically increasing on $(0,\pi/2]$, we conclude with $d/\psi<1$,~\eqref{eq:defpsi},~\eqref{eq:thetalR} and~\eqref{eq:kappasinthetadpsi} that
$$2\kappa\sin\!\left(\frac\theta2\right)
\leq\frac{4\pi^2}{L^2}d
<\frac{4\pi^2}{L^2}\psi
=\frac{4\pi^2}{L^2}\frac{L}\pi\sin\!\left(\frac{\pi l}{L}\right)
<\frac{4\pi}{L}\sin\!\left(\frac\theta2\right)\,.$$
Since $\sin(\theta/2)>0$ for $\theta\in(0,\pi]$, we can divide by it to obtain with Cauchy--Schwarz,
\begin{align}\label{eq:kappaeq1R}
\kappa<\frac{2\pi}{L}\leq\frac1{2\pi}\intst\kappa^2\,ds_t\,.
\end{align}
The definition~\eqref{eq:deftheta} of $\theta$ and~\eqref{eq:thetalR} imply
\begin{align}\label{eq:thetah}
-h\int_p^q\kappa\,ds_t+\int_p^q\kappa^2\,ds_t
\geq-h\theta+\frac{\theta^2}l
>-\theta\left(h-\frac{2\pi}{L}\right).
\end{align}
We use the evolution equation~\eqref{eq:ccf} and Lemma~\ref{lem:evolutionequations} to differentiate the ratio in time and obtain by~\eqref{eq:sinthetasinpilL},~\eqref{eq:costhetacosdpsi},~\eqref{eq:wkappasintheta} and~\eqref{eq:thetah},
\begin{align}\label{eq:dtdpsi}
&\fracp{}{t}\!\left(\frac d\psi\right) \notag\\
&\;=\frac1\psi\big(h\langle\bw,\nnu_q-\nnu_p\rangle+\langle\bw,\bkappa_q-\bkappa_p\rangle\big) 
		-\frac d{\psi^2}\cos\!\left(\frac{\pi l}{L}\right)
		\left(h\int_p^q\kappa\,ds_t-\int_p^q\kappa^2\,ds_t\right) \notag\\
	&\;\quad+\frac d{\pi\psi^2}\left(\intst\kappa^2\,ds_t-2\pi h\right)
			\left(\sin\!\left(\frac{\pi l}{L}\right)-\frac{\pi l}{L}\cos\!\left(\frac{\pi l}{L}\right)\right)\notag\\
&\;\geq\frac2\psi(h-\kappa)\sin\!\left(\frac\theta2\right) 
		-\frac\theta\psi\cos\!\left(\frac\theta2\right)\left(h-\frac{\theta}{l}\right)\notag\\
	&\;\quad+\frac1\psi\left(\frac1{\pi}\intst\kappa^2\,ds_t-2h\right)
			\left(\sin\!\left(\frac\theta2\right)-\frac{\pi l}L\cos\!\left(\frac\theta2\right)\right) \notag\\
&\;=\frac2\psi\left(\frac1{2\pi}\intst\kappa^2\,ds_t-\kappa\right)\left(\sin\!\left(\frac\theta2\right)
		-\frac\theta2\cos\!\left(\frac\theta2\right)\right) \notag\\
	&\;\quad+\frac\theta\psi\cos\!\left(\frac\theta2\right)
		\left(\frac1{2\pi}\intst\kappa^2\,ds_t-\kappa
			-h+\frac{\theta}{l}-\frac l{L\theta}\intst\kappa^2\,ds_t+\frac{h2\pi l}{L\theta}\right)\,,
\end{align}
where we just added a zero in the last step.
Since $h$ satisfies~\eqref{eq:h_1}, that is, $h\leq\intst\kappa^2\,ds_t/2\pi+2\pi/L$, we estimate with~\eqref{eq:thetalR} and~\eqref{eq:kappaeq1R},
\begin{align}\label{eq:dpsiA}
&\frac1{2\pi}\intst\kappa^2\,ds_t-\kappa
	-h+\frac{\theta}{l}-\frac l{L\theta}\intst\kappa^2\,ds_t+\frac{h2\pi l}{L\theta} \notag\\
&\quad>\left(\frac1{2\pi}\intst\kappa^2\,ds_t-h\right)\left(1-\frac{2\pi l}{L\theta}\right)
	+\frac\theta l\left(1-\frac{2\pi l}{L\theta}\right) \notag\\
&\quad>\left(\frac1{2\pi}\intst\kappa^2\,ds_t+\frac{2\pi}L-h\right)\left(1-\frac{2\pi l}{L\theta}\right)\geq0\,.
\end{align}
Furthermore, the inequality 
\begin{align}\label{eq:proofdtdpsicase2asintheta}
\sin\!\left(\frac\theta2\right)>\frac\theta2\cos\!\left(\frac\theta2\right)
\end{align}
holds for all $\theta\in(0,\pi]$.
Thus, we conclude with,~\eqref{eq:kappaeq1R},~\eqref{eq:dtdpsi},~\eqref{eq:dpsiA} and~\eqref{eq:proofdtdpsicase2asintheta} that
$$\fracp{}{t}\!\left(\frac d\psi\right)>0$$
at $(p,q,t_0)$.\\

(ii) Assume that $\langle\bw,\nnu_p\rangle=-\langle\bw,\nnu_q\rangle=\sin(\beta/2)$ and $-\beta=\theta\in(-\pi,0)$.
Again by continuity of $d/\psi$, there exists $t_1(\Si_0)\in(0,T')$ so that 
\begin{align}\label{eq:Cstar1}
\min_{\Sp\!\times\Sp}\frac d\psi(\,\cdot\,,\,\cdot\,,t)\geq\frac12\min_{\Sp\!\times\Sp}\frac d\psi(\,\cdot\,,\,\cdot\,,0)
\end{align}
for all $t\in[0,t_1]$.
For $t_0\in(t_1,T)$,
Theorem~\ref{thm:dtdstheta} applied to the initial time $t_1$ and~\eqref{eq:thetainteral2} yield $-\pi<\theta_{\min}(t_1)<\theta_{\min}(t_0)<0$ so that the monotone behaviour of the cosine on $(-\pi,0)$ implies
\begin{align}\label{eq:costhetatthetat0}
\cos\!\left(\frac{\theta_{\min}(t_0)}2\right)>\cos\!\left(\frac{\theta_{\min}(t_1)}2\right)>0\,.
\end{align}
From $0<l\leq L/2$ it follows that $1>\cos(\pi l/\!L)\geq0$ so that, by~\eqref{eq:wtaucostheta},~\eqref{eq:costhetatthetat0} and again the monotone behaviour of the cosine on $(-\pi,0)$,
\begin{align}\label{eq:Cstar2}
\frac d\psi>\frac d\psi\cos\!\left(\frac{\pi l}{L}\right)=\cos\!\left(\frac\theta2\right)
\geq\cos\!\left(\frac{\theta_{\min}(t_0)}2\right)
>\cos\!\left(\frac{\theta_{\min}(t_1)}2\right)\,.
\end{align}
With~\eqref{eq:Cstar1} and~\eqref{eq:Cstar2}, we deduce that
$$\frac d\psi
\geq\min\left\{\frac12\min_{\Sp\!\times\Sp}\frac d\psi(\,\cdot\,,\,\cdot\,,0),\cos\!\left(\frac{\theta_{\min}(t_1)}2\right)\right\}=:c(\Si_0)>0$$
at $(p,q,t_0)$. \\

(iii) Assume that $\theta\in\{0\}\cup(\pi,3\pi)$ or $\theta=\pi$ and $\langle\bw,\nnu_p\rangle=-\langle\bw,\nnu_q\rangle=1$. 
By~\eqref{eq:wtaucostheta}, $\langle\bw,\btau_q\rangle=\langle\bw,\btau_p\rangle\in[0,1)$.
Since $\Si_{t_0}$ is closed and $X(\,\cdot\,,t_0)$ is continuous, $\Si_{t_0}$ has to cross the straight line segment between $X(p,t_0)$ and $X(q,t_0)$ at least once.
Lemma~\ref{lem:mindpsitheta2pik} implies that the ratio $d/\psi$ cannot have a global minimum at $(p,q,t_0)$ (it could still, however, attain a local minimum a this point).
Hence, 
$$\frac d\psi>\min_{\Sp\!\times\Sp}\frac d\psi(\,\cdot\,,\,\cdot\,,t_0)\,,$$
where $\min_{\Sp\!\times\Sp}(d/\psi)(\,\cdot\,,\,\cdot\,,t_0)$ is attained at a point which was treated in cases~(i) and~(ii). \\

Assume that $d/\psi$ falls below $c$ and attains $\Lambda\in(0,c)$ for the first time at time $t_2\in(0,T)$ and points $p,q\in\Sp$, $p\ne q$, so that
\begin{align}\label{eq:c1mindpsi}
c>\Lambda=\frac d\psi(p,q,t_2)=\min_{\Sp\!\times\Sp}\frac d\psi(\,\cdot\,,\,\cdot\,,t_2)
\end{align}
is a global minimum and
\begin{align}\label{eq:dtdpsileq0}
\fracp{}{t}_{|_{t=t_2}}\!\left(\frac d\psi\right)(p,q,t)\leq0\,.
\end{align}
Case~(i) contradicts~\eqref{eq:dtdpsileq0}, and cases~(ii) and~(iii) contradict~\eqref{eq:c1mindpsi}. 
\end{proof}

\begin{Cor}
\label{cor:embeddedness}
Let $\Si_0=X_0(\Sp)$ be a smooth, embedded curve satisfying~\eqref{eq:intkappageqminuspi}.
Let $X:\Sp\!\times[0,T)\to\R^2$ be a solution of~\eqref{eq:ccf} satisfying~\eqref{eq:h_1} and with initial curve $\Si_0$.
Then $\St=X(\Sp\!,t)$ is embedded for all $t\in(0,T)$.
\end{Cor}

The next example shows, why the condition $\min\theta_0\geq-\pi$ is sharp.

\begin{Coex}\label{ex:cexample}
Gage~\cite[p.~53]{Gage86} suggested the following counterexample. 
Pihan~\cite[Section~5.4]{Pihan98} gave an incomplete proof for its validity which we will fix here.
If we allow local total curvature smaller than $-\pi$, then there exist counterexamples for any given minimum $\min\theta_0<-\pi$.
For the curve in Figure~\ref{fig:cexample}, 
$\theta_{\min}=\theta(p_1,p_2)<-\pi$.
We will construct a solution of~\eqref{eq:ccf} with embedded initial curve $\Si_0$ that intersects itself in finite time.
Fix $K_0>0$.
Let $\cS$ be the set of all smooth, embedded curves in $\R^2$ that satisfy
\begin{align}\label{eq:FC3alpha0}
\min\theta_0<-\pi\,,\quad
\Vert X_0\Vert_{C^{3,\alpha}(\Sp)}\leq K_0
\quad\text{ and }\quad
L(\Si)=L_0\geq8\pi K_0\,,
\end{align}
where $L_0$ is chosen big enough so that curves like in Figure~\ref{fig:cexample} are in $\cS$.
By the short time existence, see~{\cite[p.~36]{Huisken87}} or~{\cite[Thms.~4.3 and Corollary~4.4]{Pihan98}},
there exists a time $T=T(K_0)$ so that
$$\Vert X\Vert_{C^{3,\alpha;1\lfloor\alpha/2\rfloor}(\Sp\!\times[0,T/2])}\leq K_1(K_0)\,.$$
In particular,
\begin{align}\label{eq:F1K1alpha}
\left|\fracp{X^1}{t}(p,t)-\fracp{X^1}{t}(p,0)\right|\leq K_1t^{\alpha/2}\,,
\end{align}
where $X^1:=\langle X,\e_1\rangle$, and, by~\eqref{eq:ccf} and~\eqref{eq:F1K1alpha},
\begin{align}\label{eq:F1K1}
-K_1t^{\alpha/2}
\leq(h(t)-\kappa(p,t))\nnu^1(p,t)-(h(0)-\kappa(p,0))\nnu^1(p,0)
\leq K_1t^{\alpha/2}
\end{align}
for all $p\in\Sp$ and for all $t\in[0,T/2]$, where $\nnu^1:=\langle\nnu,\e_1\rangle$.
Assume $h(0)>0$ and set
\begin{align}\label{eq:t1}
t_1=t_1(K_0):=\min\left\{\frac T2,\left(\frac{h(0)}{2K_1}\right)^{-\alpha/2}\right\}\,.
\end{align}
Then~\eqref{eq:F1K1} holds for $t\in[0,t_1]$.
Let $\Si\in\cS$ be a curve like in Figure~\ref{fig:cexample}, which is symmetric about the $x_2$-axis.
Let $p,q\in\Sp$ be located as in the picture so that
\begin{align}\label{eq:nukappa0}
\nnu(p,0)=-\nnu(q,0)=-\e_1\,\qquad\text{ and }\qquad
\kappa(p,0)=\kappa(q,0)=0\,.
\end{align}
We estimate with~\eqref{eq:ccf},~\eqref{eq:F1K1},~\eqref{eq:t1} and~\eqref{eq:nukappa0},
\begin{align}\label{eq:dtF1leq_1}
\fracp{X^1}{t}(p,t)
&=(h(t)-\kappa(p,t))\nnu^1(p,t) 
\leq(h(0)-\kappa(p,0))\nnu^1(p,0)+K_1t_1^{\alpha/2} \notag\\
&\leq-h(0)+\frac{h(0)}2
=-\frac{h(0)}2
\end{align}
and likewise
\begin{align}\label{eq:dtF1leq_2}
\fracp{X^1}{t}(q,t)
&=(h(t)-\kappa(q,t))\nnu^1(q,t) 
\geq(h(0)-\kappa(q,0))\nnu^1(q,0)-K_1t_1^{\alpha/2} \notag\\
&\geq h(0)-\frac{h(0)}2
=\frac{h(0)}2
\end{align}
for $t\in[0,t_1]$.
Since $\min\theta_0<-\pi$, we can smoothly deform a curve like in Figure~\ref{fig:cexample} to achieve arbitrarily small distance between $X(p,0)$ and $X(q,0)$ without exceeding the upper bound $K_0$ in~\eqref{eq:FC3alpha0} or changing the length or enclosed area.
Hence, we can choose an embedded initial curve $\Si_0$ with
\begin{align}\label{eq:F10}
X^1(p,0)=-X^1(q,0)\leq\frac{h(0)t_1}4\,.
\end{align}
Then, by~\eqref{eq:dtF1leq_1} and~\eqref{eq:F10}
$$X^1(p,t_1)=X^1(p,0)+\int_0^{t_1}\fracp{X^1}{t}(p,t)\,dt 
\leq\frac{h(0)t_1}4-\frac{h(0)t_1}2<0$$
and by~\eqref{eq:dtF1leq_2} and~\eqref{eq:F10}
$$X^1(q,t_1)=X^1(q,0)+\int_0^{t_1}\fracp{X^1}{t}(q,t)\,dt 
\geq-\frac{h(0)t_1}4+\frac{h(0)t_1}2>0$$
so that the curve has crossed itself by the time $t_1$.
\end{Coex}

\section{Singularity analysis}\label{sec:singana}

Proposition~\ref{prop:T<infty} states that the curvature blows up if $T<\infty$.
In this section, we assume $T<\infty$ and investigate curvature blow-ups for embedded flows~\eqref{eq:ccf} that satisfy~\eqref{eq:intkappageqminuspi},~\eqref{eq:h_1} and~\eqref{eq:h_2}.
We adapt techniques from the theory of CSF to show that the curvature does not blow up in finite time and conclude $T=\infty$. \\

Proposition~\ref{prop:T<infty} motivates the following definition.
We say that a solution $X:\Do\times[0,T)\to\R^2$ of~\eqref{eq:ccf} develops a singularity at $T\leq\infty$ if $\max_{p\in\Sp}|\kappa(p,t)|\to\infty$ for $t\nearrow T$.

\begin{Lemma}
\label{lem:lowerblowuprate}
Let $X:\Do\times[0,T)\to\R^2$ be a solution of~\eqref{eq:ccf} satisfying~\eqref{eq:h_2} and with maximal time 
$T<\infty$.
Then, for all $t\in(0,T)$,
$$\max_{p\in\Sp}|\kappa(p,t)|\geq\frac1{2\sqrt{T-t}}\,.$$
\end{Lemma}

\begin{proof}
The proof is as in~\cite[Lem.~1.2]{Huisken90}, see also~\cite[Prop.~4.1]{Maeder15} or~\cite[Lem.~9.5]{Dittberner18}.
\end{proof}

Like for CSF, we distinguish between two kinds of singularities according to the blow-up rate from Lemma~\ref{lem:lowerblowuprate}. 
Let $X:\Do\times(0,T)\to\R^2$ be a solution of~\eqref{eq:ccf} with $T<\infty$.
We say that a singularity is of type I, if there exists a constant $C_0>0$ so that
$$\max_{p\in\Sp}|\kappa(p,t)|\leq\frac{C_0}{\sqrt{T-t}}$$
for all $t\in(0,T)$.
A singularity is said to be of type II, if such a constant does not exist, that is,
$$\limsup_{t\to T}\max_{p\in\Sp}|\kappa(p,t)|\sqrt{T-t}=\infty\,.$$
Type-I singularities have already been exploited in~\cite[Section~4]{Maeder15}.
We refer also to~\cite[Section~11]{White97} for a characterisation of singularities for almost Brakke flows with bounded global terms, using a monotonicity formula and a result of~\cite{Ilmanen95}.

\begin{Thm}[M\"ader-Baumdicker~{\cite[Prop.~4.12]{Maeder15}}]
\label{thm:notypeIsing}
Let $X:\Do\times(0,T)\to\R^2$ be a smooth, embedded solution of~\eqref{eq:ccf} satisfying~\eqref{eq:h_2} and with $T<\infty$. 
Then a type-I singularity cannot form at $T$.
\end{Thm}

\begin{proof}
In~\cite[Prop.~4.12]{Maeder15} the theorem is only stated for convex curves. 
But the proof does not use the convexity, see also~\cite[Section~9.4]{Dittberner18}.
By Corollaries~\ref{cor:embeddednessdl} and~\ref{cor:embeddedness}, initially embedded curves stay embedded. 
Since the global term is bounded, it will vanish in any limit flow of a type-I rescaling where we rescale by the maximal curvature.
Also, since the lengths of the curves are bounded away from zero, the curves of any limit flow will be of infinite length.
Like in the analysis in~\cite{Huisken90} of type-I singularities of mean curvature flow, a monotonicity formula, see~{\cite[Proposition~4.9]{Maeder15}} or~\cite[Theorem 8.5]{Dittberner18}, yields that any limit flow of a type-I rescaling is an embedded homothetically shrinking solution of CSF with non-vanishing curvature.
By~\cite{AbreschLanger86}, this is an embedded shrinking circle.
This contradicts the unbounded length.
\end{proof}

To investigate type-II singularities, we want to rescale the curves $\St$ near a singular point as $t\to T<\infty$. 
The following rescaling technique for type-II singularities was introduced in~\cite[Proof of Thm.~16.4]{Hamilton95b} for Ricci flow, and applied to type-II singularities of MCF in~\cite[p.~11]{HuiskenSinestrari99}.
Let $(p_k,t_k)_{k\in\N}$ be a sequence in $\Do\times[0,T-1/k]$ with
$$T_k:=\kappa^2(p_k,t_k)\left(T-\frac1k-t_k\right)=\max_{(p,t)\in\Sp\!\times[0,T-1/k]}\left(\kappa^2(p,t)\left(T-\frac1k-t\right)\right)$$
for each $k\in\N$.
We set $\lambda_k^2:=\kappa^2(p_k,t_k)$, $\alpha_k:=-\lambda_k^2t_k$ and define the rescaled embeddings $X_k:\Do\times[\alpha_k,T_k]\to\R^2$ by
\begin{align}\label{eq:FkII}
X_k(p,\tau):=\lambda_k\left(X\!\left(p,t_k+\frac\tau{\lambda_k^2}\right)-X(p_k,t_k)\right)\,.
\end{align}

\begin{Thm}
\label{thm:Finftycsf}
Let $X:\Do\times(0,T)\to\R^2$ be a smooth, embedded solution of~\eqref{eq:ccf} with $T<\infty$ and satisfying~\eqref{eq:intabskappa} for $\Do=\R$ and~\eqref{eq:h_2} for $\Do\in\{\Sp,\R\}$.
Then there exists a sequence of intervals $0\in I_k\subset\R$ and rescaled embeddings 
$$\left(\bar X_{k}:I_k\times[\alpha_k,T_k]\to\R^2\right)_{k\in\N}$$
that converges for $k\to\infty$ along a subsequence, uniformly and smoothly on compact subsets $I\times J\subset\R\times\R$ with $0\in I$ and compact subsets in $\R^2$ to a maximal, smooth, strictly convex or strictly concave limit solution $X_{\infty}:\R\times\R\to\R^2$ which satisfies
$$\fracp{X_{\infty}}{\tau}(p,\tau)=-\kappa_{\infty}(p,\tau)\nnu_{\infty}(p,\tau)\,.$$
Moreover, $L(\Si_\tau^{\infty})=\infty$ for all $\tau\in\R$, $X_{\infty}(0,0)=0$, $\sup_{\R\times\R}|\kappa_{\infty}|=|\kappa_{\infty}(0,0)|=1$. 
\end{Thm}

\begin{proof}
The convergence follows similar lines to those of~\cite[Rem.~4.22(2)]{Ecker04} and~\cite[Prop.~4.7]{Maeder15}.
For details, see also~\cite[Thm.~9.13]{Dittberner18}.
The strict convexity/concavity is proofed like in~\cite[Thms.~5.14 and~7.7]{Altschuler91}, where we use that, away from $T$, the coefficients in the evolution equation for the curvature are bounded and that $\intst|\kappa|\,ds_t<\infty$ (see property~\eqref{eq:intabskappa} for $\Do=\R$).
A more detailed proof can be found in~\cite[Prop.~4.3.2]{Mantegazza11} or in~\cite[Prop.~9.16]{Dittberner18}.
\end{proof}

We now can proceed as in~\cite[Thm.~2.4]{Huisken95}.

\begin{Thm}
\label{thm:notypeIIsing}
Let $\Si_0=X_0(\Do)$ be a smooth, embedded curve satisfying~\eqref{eq:alpha} and~\eqref{eq:intabskappa} for $\Do=\R$ as well as~\eqref{eq:intkappageqminuspi},~\eqref{eq:h_1} and~\eqref{eq:h_2} for $\Do\in\{\Sp,\R\}$.
Let $X:\Do\times[0,T)\to\R^2$ be a solution of~\eqref{eq:ccf} with $T<\infty$ and initial curve $\Si_0$.
Then a type-II singularity cannot form at~$T$.
\end{Thm}

\begin{proof}
Theorem~\ref{thm:Finftycsf} yields that the limit flow consists of strictly convex or concave curves $\Si^{\infty}_\tau$ for $\tau\in\R$ satisfying $\sup_{\R\times\R}|\kappa_{\infty}|=|\kappa_{\infty}(0,0)|=1$.
If $\kappa_{\infty}<0$, we change the direction of parametrisation so that $\kappa_{\infty}>0$.
Since the curvature attains its maximum at the point $(0,0)\in\R\times\R$,~\cite[Main Theorem~B]{Hamilton95a} yields that $X_{\infty}$ is a translating solution of CSF.
~\cite[Thm.~8.16]{Altschuler91} implies that $\Si^{\infty}_\tau$ is the grim reaper for every $\tau\in\R$.
The grim reaper is asymptotic to two parallel lines of distance $\pi$ from inside.
Let $\tau\in\R$.
We can find a sequence of points $(p_j,q_j)_{j\in\N}$ in $\R\times\R$ with $d_{\infty}(p_j,q_j,\tau)\leq\pi$ for all $j\in\N$ and $l_{\infty}(p_j,q_j,\tau)\to\infty$ for $j\to\infty$.
Hence,
$$\inf_{\R\times\R}\frac{d_{\infty}}{l_{\infty}}(\,\cdot\,,\,\cdot\,,\tau)=0\,.$$
However, like in~\cite[Thms.~2.4 and~2.5]{Huisken95} (for details, see~\cite[Thm.~9.21]{Dittberner18}), the lower bound $\inf_{\R\times\R\times[0,T)}(d/l)\geq c$ from Theorem~\ref{thm:minimumdl}, respectively the lower bound $\inf_{\Sp\!\times\Sp\!\times[0,T)}(d/\psi)\geq c$ from Theorem~\ref{thm:minimumdpsi} imply that
$$\inf_{\R\times\R}\frac{d_{\infty}}{l_{\infty}}(\,\cdot\,,\,\cdot\,,\tau)\geq c$$
for every limit flow of rescalings according to~\eqref{eq:FkII}.
\end{proof}

\begin{Cor}\label{cor:T=infty}
Let $\Si_0=X_0(\Do)$ be a smooth, embedded curve satisfying~\eqref{eq:alpha} and~\eqref{eq:intabskappa} for $\Do=\R$ as well as~\eqref{eq:intkappageqminuspi},~\eqref{eq:h_1} and~\eqref{eq:h_2} for $\Do\in\{\Sp,\R\}$.
Let $X:\Do\times[0,T)\to\R^2$ be a solution of~\eqref{eq:ccf} with initial curve $\Si_0$.
Then $T=\infty$.
\end{Cor}

\begin{proof}
By Theorems~\ref{thm:notypeIsing} and~\ref{thm:notypeIIsing} neither a type-I nor a type-II singularity can form at $T$ so that curvature stays bounded on $[0,T]$ by a constant $C(\Si_0,T)$.
We can extend the flow beyond $T$ and repeat the above argument.
Hence, for every time $T'<\infty$, there exists a constant $C(\Si_0,T')<\infty$ so that $\max_{p\in\Sp}|\kappa(p,t)|\leq C$ for all $t\in[0,T')$.
Applying Proposition~\ref{prop:T<infty} yields that the short time solution can be extended to a smooth solution on $(0,\infty)$.
\end{proof}

\section{Convexity in finite time}\label{sec:convexity}

In this section, we show that a smooth, embedded solution $X:\Sp\!\times(0,\infty)\to\R^2$ of~\eqref{eq:ccf} with a global term $h$ satisfying~\eqref{eq:h_3a} becomes convex in finite time.

\begin{Rem}\label{rem:h}
We observe that, by Lemma~\ref{lem:dtL},
\begin{align}\label{eq:hdtAdtL}
h=\frac1L\left(2\pi+\fracd{A}{t}\right)
=\frac1{2\pi}\left(\intst\kappa^2\,ds_t+\fracd{L}{t}\right)
\end{align}
and
\begin{align}\label{eq:L2piintkappa}
-\fracd{}{t}\left(\frac{L^2}{4\pi}-A\right)
=\frac L{2\pi}\intst\kappa^2\,ds_t-2\pi\,.
\end{align}
In respect of~\eqref{eq:L2piintkappa}, choose
\begin{align}\label{eq:dtAa}
\fracd{A}{t}=\gamma\left(\frac L{2\pi}\intst\kappa^2\,ds_t-2\pi\right)
\end{align}
and
\begin{align}\label{eq:dtLa}
\frac L{2\pi}\fracd{L}{t}
=\frac1{4\pi}\fracd{L^2}{t}
=-(1-\gamma)\left(\frac L{2\pi}\intst\kappa^2\,ds_t-2\pi\right)\,,
\end{align}
where $\gamma\in\R$.
Then,~\eqref{eq:hdtAdtL} yields
$$h=(1-\gamma)\frac{2\pi}{L}+\frac{\gamma}{2\pi}\intst\kappa^2\,ds_t\,.$$
For arbitrary $\gamma<0$, however, the positivity of $h$ is not guaranteed.
\end{Rem}

\begin{Lemma}\label{lem:AL}
Let $X:\Sp\!\times[0,T)\to\R^2$ be a solution of~\eqref{eq:ccf} with initial curve $\Si_0$ and $h$ satisfying~\eqref{eq:h_3a}. 
Then, $A$ and $L$ are monotone and there exist constants $0<c<C<\infty$ such that $c\leq A,L\leq C$ on $[0,T)$ and
$$\frac{L}{2\pi}\int_{\Si_\tau}\kappa^2\,ds_\tau-2\pi\in L^1([0,T))\,.$$
\end{Lemma}

\begin{proof}
By~\eqref{eq:dtAa} and~\eqref{eq:dtLa},
\begin{align}\label{eq:dtAdtL2gamma}
(1-\gamma)\fracd{A}{t}=-\frac\gamma{4\pi}\fracd{L^2}{t}
\end{align}
so that, with $\delta\in(0,\infty)$ and
$$\gamma=\frac{(\delta-1)A_0}{L_0^2/4\pi-A_0}\in\left(-\frac{A_0}{L_0^2/4\pi-A_0},\infty\right)\,,$$
integrating~\eqref{eq:dtAdtL2gamma} yields
\begin{align}\label{eq:AL2gamma}
(1-\gamma)A_t+\gamma\frac{L^2_t}{4\pi}
=(1-\gamma)A_0+\gamma\frac{L^2_0}{4\pi}
=\delta A_0
\end{align}
for all $t\in(0,T)$.
For $\delta\in(0,1)$, we have $\gamma<0$ and $-(1-\gamma)<0$, so that by~\eqref{eq:dtAa} and~\eqref{eq:dtLa},
$$\fracd{A}{t}<0
\qquad\text{ and }\qquad
\fracd{L}{t}<0\,.$$
Hence, $A$ and $L$ are uniformly bounded away from infinity.
By~\eqref{eq:AL2gamma},
$$(1-\gamma)A_t
>(1-\gamma)A_t+\gamma\frac{L^2_t}{4\pi}
=\delta A_0$$
and so that, by the isoperimetric inequality, $A$ and $L$ are uniformly bounded away from zero.
For $\delta\in[1,L_0^2/4\pi A_0]$, we have $\gamma\in[0,1]$ and by~\eqref{eq:dtAa} and~\eqref{eq:dtLa},
$$\fracd{A}{t}\geq0
\qquad\text{ and }\qquad
\fracd{L}{t}\leq0\,.$$
Hence, $A$ and $L$ are uniformly bounded away from zero and infinity.
For $\delta>L_0^2/4\pi A_0$, we have $\gamma>1$ and by~\eqref{eq:dtAa} and~\eqref{eq:dtLa},
$$\fracd{A}{t}>0
\qquad\text{ and }\qquad
\fracd{L}{t}>0\,.$$
Hence, $A$ and $L$ are uniformly bounded away from zero.
By~\eqref{eq:AL2gamma},
$$\gamma\frac{L^2_t}{4\pi}
<(1-\gamma)A_t+\gamma\frac{L^2_t}{4\pi}
=\delta A_0$$
and so that $A$ and $L$ are uniformly bounded away from infinity. 
The uniform bounds on the area and length from above and~\eqref{eq:L2piintkappa} yield
\[\frac{L}{2\pi}\int_{\Si_\tau}\kappa^2\,ds_\tau-2\pi\,d\tau\in L^1([0,T))\,.\qedhere\]
\end{proof}

Like in~\cite[Section~7]{Maeder15}, we use the following Gagliardo--Nirenberg interpolation inequality.

\begin{Thm}[Gagliardo--Nirenberg interpolation inequality,~{\cite[pp.~125]{Nirenberg59}}, see also~{\cite[Thm.~3.70]{Aubin98}}]\label{thm:GN}
Let $f\in C^\infty(\Sp)$. 
Let $p>2$ and $\sigma\in[0,1)$ with $\sigma=1/2-1/p$.
Then there exist constants $C_1=c_1(p,\sigma)$ and $C_2=c_2(p,\sigma)$ such that
\begin{align*}
\left(\intsp|f|^p\,dx\right)^{1/p}
&\leq C_1\left(\intsp\left(\fracd{f}{x}\right)^{2}dx\right)^{\sigma/2}
	\left(\intsp f^2\,dx\right)^{(1-\sigma)/2} \notag\\
	&\quad+C_2\left(\intsp f^2\,dx\right)^{1/2}\,.
\end{align*}
\end{Thm}

\begin{Lemma}[see proof of~{\cite[Cor.~7.5]{Maeder15}}]\label{lem:dtfintffto0}
Let $f\in C^1((0,\infty))\cap L^1((0,\infty))$ with $f\geq0$ and $\fracd{}{t}f\leq C(C+f)^3$ for $C\geq0$. 
Then $f(t)\to0$ for $t\to\infty$.
\end{Lemma}

\begin{Lemma}\label{lem:dthlp}
Let $X:\Sp\!\times(0,\infty)\to\R^2$ be a smooth, embedded solution of~\eqref{eq:ccf} and $h$ satisfying~\eqref{eq:h_3a}. 
Then there exists a constant $C>0$ such that, for all $t\in(0,\infty)$,
$$\fracd{}{t}\left(\frac{L}{2\pi}\intst\kappa^2\,ds_t-2\pi\right)
\leq C\left(C+\frac{L}{2\pi}\intst\kappa^2\,ds_t-2\pi\right)^3\,.$$
\end{Lemma}

\begin{proof}
By Lemma~\ref{lem:AL}, we can estimate
\begin{align}\label{eq:dtintkappa_33}
h\leq C\left(1+\intst\kappa^2\,ds_t\right)\,,
\end{align}
where $C$ depends on $\gamma$ and the lower bound on $L$. 
We deduce with Theorem~\ref{thm:GN} for $p=4$ and $\sigma=1/4$, the estimate $(a+b)^4\leq C(a^4+b^4)$, and Young's inequality for $p=q=2$,
\begin{align}\label{eq:dtintkappa_4}
\intst\kappa^4\,ds_t
&\leq\left(C\left(\intst\left(\fracp{\kappa}{s}\right)^{2}\,ds_t\right)^{1/8}
		\left(\intst\kappa^2\,ds_t\right)^{3/8}
		+C\left(\intst\kappa^2\,ds_t\right)^{1/2}\right)^4 \notag\\
&\leq\delta\intst\left(\fracp{\kappa}{s}\right)^{2}\,ds_t
		+C(\delta)\left(\intst\kappa^2\,ds_t\right)^3
		+C\left(\intst\kappa^2\,ds_t\right)^{2}
\end{align}
for a constants $C>0$.
Again, by Theorem~\ref{thm:GN} for $p=3$ and $\sigma=1/6$, the estimate $(a+b)^3\leq C(a^3+b^3)$,  and Young's inequality for $p=4$ and $q=4/3$, 
\begin{align}
&\intst\kappa^3\,ds_t \notag\\
&\quad\leq\left(C\left(\intst\left(\fracp{\kappa}{s}\right)^{2}\,ds_t\right)^{1/12}
		\left(\intst\kappa^2\,ds_t\right)^{5/12}
		+C\left(\intst\kappa^2\,ds_t\right)^{1/2}\right)^3\label{eq:dtintkappa_3}\\
&\quad\leq\delta\intst\left(\fracp{\kappa}{s}\right)^{2}\,ds_t
		+C(\delta)\left(\intst\kappa^2\,ds_t\right)^{5/3}
		+C\left(\intst\kappa^2\,ds_t\right)^{3/2}\label{eq:dtintkappa_3b}\,.
\end{align}
Multiplying $\intst\kappa^2\,ds_t$ to~\eqref{eq:dtintkappa_3} yields with Young's inequality for $p=4$ and $q=4/3$ that
\begin{align}\label{eq:dtintkappa_32}
&\intst\kappa^2\,ds_t\intst\kappa^3\,ds_t \notag\\
&\quad\leq\left(C\left(\intst\left(\fracp{\kappa}{s}\right)^{2}\,ds_t\right)^{1/12}
		\left(\intst\kappa^2\,ds_t\right)^{9/12}
		+C\left(\intst\kappa^2\,ds_t\right)^{5/6}\right)^3\notag\\
&\quad\leq\delta\intst\left(\fracp{\kappa}{s}\right)^{2}\,ds_t
		+C(\delta)\left(\intst\kappa^2\,ds_t\right)^3 
		+C\left(\intst\kappa^2\,ds_t\right)^{5/2}.
\end{align}
We use Lemma~\ref{lem:evolutionequations}, integration by parts,~\eqref{eq:dtintkappa_33}, and~\eqref{eq:dtintkappa_4},~\eqref{eq:dtintkappa_3b},~\eqref{eq:dtintkappa_32} with $\delta=1/3$ to calculate,
\begin{align}\label{eq:dtintkappa2}
\fracd{}{t}\intst\kappa^2\,ds_t
&=2\intst\left(\kappa\fracp{^2\kappa}{s^2}-(h-\kappa)\kappa^3\right)ds_t
		+\intst\kappa^3(h-\kappa)\,ds_t \notag\\
&=-2\intst\left(\fracp{\kappa}{s}\right)^{2}ds_t
		-h\intst\kappa^3\,ds_t
		+\intst\kappa^4\,ds_t \notag\\
&\leq C\sum_{p=2,\frac32,\frac53,\frac52,3}\left(\intst\kappa^2\,ds_t\right)^p
\leq C\left(1+\intst\kappa^2\,ds_t\right)^3
\end{align}
for all $t\in(0,\infty)$.
By Lemma~\ref{lem:dtL}, the bounds on $L$ from Lemma~\ref{lem:AL},~\eqref{eq:dtintkappa_33} and~\eqref{eq:dtintkappa2},
\begin{align*}
&\fracd{}{t}\left(\frac{L}{2\pi}\intst\kappa^2\,ds_t-2\pi\right)
=\frac1{2\pi}\fracd{L}{t}\intst\kappa^2\,ds_t
		+\frac L{2\pi}\fracd{}{t}\intst\kappa^2\,ds_t \notag\\
&\quad\leq \frac1{2\pi}\left(2\pi h-\intst\kappa^2\,ds_t\right)\intst\kappa^2\,ds_t
	+C\left(1+\intst\kappa^2\,ds_t\right)^3 \notag\\
&\quad\leq C\left(1+\intst\kappa^2\,ds_t\right)^2
		+C\left(1+\intst\kappa^2\,ds_t\right)^3 \notag\\
&\quad\leq C\left(C+\frac{L}{2\pi}\intst\kappa^2\,ds_t-2\pi\right)^3\,.\qedhere
\end{align*}
\end{proof}

\begin{Lemma}\label{lem:inthkappa}
Let $X:\Sp\!\times[0,\infty)\to\R^2$ be a smooth, embedded solution of~\eqref{eq:ccf} with initial curve $\Si_0$ and $h$ satisfying~\eqref{eq:h_3a}.
Then, 
$$Lh\to2\pi
\qquad\text{ and }\qquad
\frac{L}{2\pi}\intst\kappa^2\,ds_t\to2\pi$$
for $t\to\infty$, and there exist a time $t_0\geq0$ and constants $0<c<C<\infty$ such that $\inf_{[t_0,\infty)}h\geq c$ and
$$\sup_{[0,\infty)}h
+\sup_{[0,\infty)}\left|\fracd{h}{t}\right|
+\sup_{[0,\infty)}\intst\kappa^2\,ds_t
\leq C\,.$$
\end{Lemma}

\begin{proof}
Lemmata~\ref{lem:AL},~\ref{lem:dtfintffto0} and~\ref{lem:dthlp} yield that
\begin{align}\label{eq:L2piintkappa2to2pi}
\frac{L}{2\pi}\intst\kappa^2\,ds_t\to2\pi
\end{align}
for $t\to\infty$.
By Lemma~\ref{lem:AL}, $L$ is bounded away from zero and infinity.
By~\eqref{eq:dtAa} 
and~\eqref{eq:L2piintkappa2to2pi},
$$\fracd{A}{t}=\intst(h-\kappa)\,ds_t=Lh-2\pi\to0$$
for $t\to\infty$.
Hence, there exist a time $t_0\in[0,\infty)$ and constants $0<c<C<\infty$ so that $c\leq h\leq C$ on $[t_0,\infty)$. 
By~\eqref{eq:dtLa}
and~\eqref{eq:L2piintkappa2to2pi},
$$\fracd{L}{t}=\intst\kappa^2\,ds_t-2\pi h\to0$$
for $t\to\infty$ so that there exist $0<C<\infty$ with $\big|\fracd{}{t}L\big|+\intst\kappa^2\,ds_t\leq C$ on $[0,\infty)$. 
This yields $\big|\fracd{h}{t}\big|\leq C$.
\end{proof}

\begin{Lemma}\label{lem:dtinthkappa2}
Let $X:\Sp\!\times(0,\infty)\to\R^2$ be a smooth, embedded solution of~\eqref{eq:ccf}. 
Then there exists a constant $C>0$ such that
\begin{align*}
\fracd{}{t}\intst(h-\kappa)^2\,ds_t
&\leq-\intst\left(\fracp{\kappa}{s}\right)^{2}\,ds_t
		+C\fracd{h}{t}\intst(h-\kappa)\,ds_t \\
	&\quad+C\sum_{i=1}^5\left(\intst(h-\kappa)^2\,ds_t\right)^{p_i}h^{q_i}
\end{align*}
for all $t\in(0,\infty)$, where $p_i\in[1,3]$ and $q_i\in[0,2]$.
\end{Lemma}

\begin{proof}
We follow the lines of~\cite[Lems.~7.3 and~7.4]{Maeder15}. 
Write $\kappa=h-(h-\kappa)$.
Then
$$(h-\kappa)^3\kappa=h(h-\kappa)^3-(h-\kappa)^4$$
and
$$(h-\kappa)^2\kappa^2 
=h^2(h-\kappa)^2-2h(h-\kappa)^3+(h-\kappa)^4\,.$$
Lemma~\ref{lem:evolutionequations} and integration by parts yields
\begin{align}\label{eq:dtinthminuskappa_1}
&\fracd{}{t}\intst(h-\kappa)^2\,ds_t \notag\\
&\,=\intst(h-\kappa)^3\kappa\,ds_t
		+2\fracd{h}{t}\intst(h-\kappa)\,ds_t 
		+2\intst(h-\kappa)
			\left(-\fracp{^2\kappa}{s^2}+(h-\kappa)\kappa^2\right)\,ds_t \notag\\
&\,=-2\intst\left(\fracp{\kappa}{s}\right)^{2}ds_t
			+2\fracd{h}{t}\intst(h-\kappa)\,ds_t
			+\intst(h-\kappa)^4\,ds_t \notag\\
	&\qquad-3h\intst(h-\kappa)^3\,ds_t 
			+2h^2\intst(h-\kappa)^2\,ds_t\,.
\end{align}
Like in~\cite[Cor.~7.4]{Maeder15}, we use Theorem~\ref{thm:GN} with $p=4$ and $\sigma=1/4$ and Young's inequality with $p=4/3$ and $q=4$ as well as for $p=q=2$, to estimate
\begin{align}\label{eq:dtinthminuskappa_2}
&\intst(h-\kappa)^4\,ds_t \notag\\
&\,\leq\left(C\left(\intst\left(\fracp{\kappa}{s}\right)^{2}\,ds_t\right)^{1/8}
		\left(\intst(h-\kappa)^2\,ds_t\right)^{3/8} 
		+C\left(\intst(h-\kappa)^2\,ds_t\right)^{1/2}\right)^4 \notag\\
&\,\leq\frac12\intst\left(\fracp{\kappa}{s}\right)^{2}\,ds_t
		+C\left(\intst(h-\kappa)^2\,ds_t\right)^3 
		+C\left(\intst(h-\kappa)^2\,ds_t\right)^{2}.
\end{align}
Again by Theorem~\ref{thm:GN} with $p=3$ and $\sigma=1/6$ and Young's inequality for $p=3/2$ and $q=3$ as well as for $p=4$ and $q=4/3$ we obtain
\begin{align}\label{eq:dtinthminuskappa_3}
3h\intst(h-\kappa)^3\,ds_t
&\leq 3h\left(C\left(\intst\left(\fracp{\kappa}{s}\right)^{2}\,ds_t\right)^{1/12}
		\left(\intst(h-\kappa)^2\,ds_t\right)^{5/12}\right. \notag\\
	&\hspace{4em}+\left.{\vphantom{C\left(\intst\left(\fracp{\kappa}{s}\right)^{2}\,ds_t\right)^{1/12}
		\left(\intst(h-\kappa)^2\,ds_t\right)^{5/12}}}{} C\left(\intst(h-\kappa)^2\,ds_t\right)^{1/2}\right)^3 \notag\\
&\leq\frac12\intst\left(\fracp{\kappa}{s}\right)^{2}\,ds_t
		+Ch^{4/3}\left(\intst(h-\kappa)^2\,ds_t\right)^{5/3} \notag\\
	&\quad+Ch\left(\intst(h-\kappa)^2\,ds_t\right)^{3/2}.
\end{align}
Altogether,~\eqref{eq:dtinthminuskappa_1},~\eqref{eq:dtinthminuskappa_2},~\eqref{eq:dtinthminuskappa_3} yield the claim.
\end{proof}

\begin{Lemma}\label{lem:dtinthkappa}
Let $X:\Sp\!\times[0,\infty)\to\R^2$ be a smooth, embedded solution of~\eqref{eq:ccf} with initial curve $\Si_0$ and $h$ satisfying~\eqref{eq:h_3a}. 
Then
$$\intst(h-\kappa)^2\,ds_t\to0$$
for $t\to\infty$ and
$$\int_0^\infty\intst(h-\kappa)^2\,ds_t\,dt+\int_0^\infty\intst\left(\fracp{\kappa}{s}\right)^{2}\,ds_t\,dt<\infty\,.$$
\end{Lemma}

\begin{proof}
Similar to~\cite[p.~47]{Huisken87}, Lemma~\ref{lem:dtL} yields
$$\intst(h-\kappa)^2\,ds_t
=Lh^2-4\pi h+\intst\kappa^2\,ds_t 
=h\fracd{A}{t}-\fracd{L}{t}\,.$$
Lemmata~\ref{lem:AL} and~\ref{lem:inthkappa} imply for $0<\vare<\tau<\infty$,
$$\int_\vare^\tau\intst(h-\kappa)^2\,ds_t\,dt
=\sup_{t\in[\vare,\tau]}h(t)(A_\tau-A_\vare)+(L_\vare-L_\tau)\leq C\,.$$
We let $\vare\to0$ and $\tau\to\infty$ to obtain
\begin{align}\label{eq:intinthkappa2}
\int_0^\infty\intst(h-\kappa)^2\,ds_t\,dt<\infty\,.
\end{align}
By Lemma~\ref{lem:inthkappa},
$$\fracd{h}{t}\intst(h-\kappa)\,ds_t\,dt
\leq\sup_{[0,\infty)}\left|\fracd{h}{t}\right||Lh-2\pi|
\leq C\,,$$
so that Lemma~\ref{lem:dtinthkappa2} implies
$$\fracd{}{t}\intst(h-\kappa)^2\,ds_t 
\leq C\left(1+\intst(h-\kappa)^2\,ds_t\right)^3\,.$$
Like in~\cite[Cor.~7.5]{Maeder15}, Lemma~\ref{lem:dtfintffto0} yields
$$\intst(h-\kappa)^2\,ds_t\to0$$
for $t\to\infty$.
Consequently, there exists a time $t_0\geq0$ so that
$$\intst(h-\kappa)^2\,ds_t<1$$
for all $t>t_0$, and thus
\begin{align}\label{eq:inthkappa3}
\left(\intst(h-\kappa)^2\,ds_t\right)^{p}\leq\intst(h-\kappa)^2\,ds_t
\end{align}
for all $p\geq1$ and $t>t_0$.
By Lemma~\ref{lem:AL}, $\fracd{}{t}A$ has a sign so that
$$\int_\vare^\tau\left|\intst(h-\kappa)\,ds_t\right|\,dt
=|A_\tau-A_\vare|
\leq C\,,$$
where $C>0$ is independent of time.
Sending $\vare\to0$ and $\tau\to\infty$ yields with Lemma~\ref{lem:inthkappa},
$$\int_0^\infty\fracd{h}{t}\intst(h-\kappa)\,ds_t\,dt
\leq\sup_{[0,\infty)}\left|\fracd{h}{t}\right|\int_0^\infty\left|\intst(h-\kappa)\,ds_t\right|\,dt
\leq C.$$
Thus, with Lemma~\ref{lem:dtinthkappa2},~\eqref{eq:intinthkappa2} and~\eqref{eq:inthkappa3} we obtain
\begin{align*}
\int_{t_0}^\infty\intst\left(\fracp{\kappa}{s}\right)^{2}\,ds_t\,dt 
\leq\int_{\Si_{t_0}}(h-\kappa)^2\,ds_t\,dt
		+C+C\int_{t_0}^\infty\intst(h-\kappa)^2\,ds_t\,dt<\infty\,.
\end{align*}
Since $\St$ is smooth for $t\in[0,t_0]$, the claim follows.
\end{proof}

\begin{Thm}
\label{thm:convexity}
Let $X:\Sp\!\times[0,\infty)\to\R^2$ be a smooth, embedded solution of~\eqref{eq:ccf} with initial curve $\Si_0$ and $h$ satisfying~\eqref{eq:h_3a}. 
Then there exists a time $T_0\geq0$ such that $\St$ is strictly convex for $t>T_0$.
\end{Thm}

\begin{proof}
By Lemma~\ref{lem:inthkappa},
\begin{align}\label{eq:hleqpiL}
h\geq c_h>0
\end{align}
on $[t_0,\infty)$ for $t_0\geq0$ and $c_h>0$.
Lemma~\ref{lem:dtinthkappa} implies that there exists a sequence $(t_k)_{k\in\N}$ with $t_k\to\infty$ for $k\to\infty$ so that
\begin{align}\label{eq:dpkappato0}
\int_{\Si_{t_k}}\left(\fracp{\kappa}{s}\right)^{2}ds_{t_k}\to0
\end{align}
for $k\to\infty$.
Hence, there exists $k_0\in\N$ so that for all $k\geq k_0$
\begin{align}\label{eq:dpkappaC1}
\int_{\Si_{t_k}}\left(\fracp{\kappa}{s}\right)^{2}ds_{t_k}<1\,.
\end{align}
We employ~{\cite[Thm.~7.26(ii)]{GilbargTrudinger83}} to obtain that $W^{1,2}(\Sp)$ is compactly embedded in $C^0(\Sp)$. 
Furthermore, $C^0(\Sp)\subset L^2(\Sp)$, and 
$\Vert f\Vert_{L^2(\Sp)}
\leq\sqrt{2\pi}\Vert f\Vert_{C^0(\Sp)}$
for every $f\in C^0(\Sp)$.
Hence, $C^0(\Sp)$ is continuously embedded in $L^2(\Sp)$. 
Let $f\in W^{1,2}(\Sp)$.
By Ehrling's lemma, for all $\vare>0$ there exists a constant $C(\vare)>0$ so that
\begin{align}\label{eq:Ehrling_W12}
\Vert f\Vert_{C^0(\Sp)}\leq\vare\Vert f\Vert_{W^{1,2}(\Sp)}+C(\vare)\Vert f\Vert_{L^2(\Sp)}\,.
\end{align}
Lemma~\ref{lem:dtinthkappa} and~\eqref{eq:dpkappato0} yield $h(t_k)-\kappa(\,\cdot\,,t_k)\in W^{1,2}(\Sp)$ for each $k\in\N$.
Hence, we can use~\eqref{eq:dpkappaC1} and~\eqref{eq:Ehrling_W12} to estimate
\begin{align}\label{eq:maxhlpkappa_0}
\max_{p\in\Sp}|h(t_k)-\kappa(p,t_k)| 
&\leq\vare\left(\int_{\Si_{t_k}}\left(\fracp{\kappa}{s}\right)^{2}\,ds_{t_k}\right)^{1/2}
		+\vare\left(\int_{\Si_{t_k}}(h-\kappa)^2\,ds_{t_k}\right)^{1/2} \notag\\
	&\quad+C(\vare)\left(\int_{\Si_{t_k}}(h-\kappa)^2\,ds_{t_k}\right)^{1/2} \notag\\
&\leq\vare+C(\vare)\left(\int_{\Si_{t_k}}(h-\kappa)^2\,ds_{t_k}\right)^{1/2}
\end{align}
for all $k\geq k_0$.
Choose $\vare=c_h/4$ to deduce with~\eqref{eq:maxhlpkappa_0}
\begin{align}\label{eq:maxhlpkappa}
\max_{p\in\Sp}|h(t_k)-\kappa(p,t_k)|
\leq\frac{c_h}4+C\left(\int_{\Si_{t_k}}(h-\kappa)^2\,ds_{t_k}\right)^{1/2}\,.
\end{align}
Lemma~\ref{lem:dtinthkappa} implies that there exists $k_1\geq k_0$ so that for all $k\geq k_1$
$$\int_{\Si_{t_k}}(h-\kappa)^2\,ds_{t_k}<\left(\frac{c_h}{4C}\right)^2\,.$$ 
By~\eqref{eq:maxhlpkappa},
$$\max_{p\in\Sp}|h(t_{k_1})-\kappa(p,t_{k_1})|\leq\frac{c_h}2\,.$$
With~\eqref{eq:hleqpiL}, we conclude that $\kappa>0$ at $t_{k_1}$. 
From Corollary~\ref{cor:strongmaxpkappa} it follows that $\kappa>0$ for all $t>t_{k_1}$. 
Hence, the claim holds for $T_0=t_{k_1}$.
\end{proof}

\section{Longtime behaviour}\label{sec:longtimebehaviour}

In this section we show that convex solutions of~\eqref{eq:ccf} that exist for all positive times converge exponentially and smoothly to a round circle.
This was already shown in~\cite{Gage86} for the APCSF and in~\cite{Pihan98} for the LPCF.
We repeat and extend the arguments here for $h$ satisfying~\eqref{eq:h_3a} 
for the sake of completeness.
We mostly follow the lines of~\cite[Section~5]{GageHamilton86} for rescaled convex CSF,~\cite{Gage86} for convex APCSF, and~\cite[Chapter~7]{Pihan98} for convex LPCF.
For further details, see~\cite[Chapter~11]{Dittberner18}.

\begin{Lemma}[Isoperimetric inequality, Gage~\cite{Gage83}]\label{lem:intkappaLA}
For a closed, convex $C^2$-curve in the plane,
$$\ints\kappa^2\,ds\geq\frac{\pi L}{A}$$
with equality if and only if the curve is a circle.
\end{Lemma}

\begin{Lemma}
\label{lem:isoleqexp}
Let $\Si_0=X_0(\Sp)$ be a smooth, embedded, convex curve.
Let $X:\Sp\!\times[0,\infty)\to\R^2$ be a solution of~\eqref{eq:ccf} with initial curve $\Si_0$.
Then there exists a constant $C=C(\Si_0)>0$, such that, for all $t>0$,
$$\left(\frac{L^2}{A}-4\pi\right)\leq C\exp\!\left(-\int_0^t\frac{2\pi}{A}\,d\tau-\log\frac{A_t}{A_0}\right)\,.$$
\end{Lemma}

\begin{proof}
We follow the lines of~{\cite[Cor.~2.4]{Gage86}} and~{\cite[Lem.~7.7]{Pihan98}} and use Lemma~\ref{lem:intkappaLA} to estimate for $t>0$
\begin{align*}
\fracd{}{t}\!\left(\frac{L^2}{A}-4\pi\right)
&=-\frac{2L}{A}\left(\intst\kappa^2\,ds_t-2\pi h\right)
	-\frac{L^2}{A^2}\left(Lh-2\pi\right) \\
&\leq-\frac{L}{A}\left(\frac{2\pi L}A-4\pi h+\frac{L^2}{A}h-\frac{2\pi L}{A}\right)
=-\frac{hL}{A}\left(\frac{L^2}{A}-4\pi\right)\,.
\end{align*}
By~\eqref{eq:hdtAdtL},
\[\int_0^t\frac{hL}{A}\,d\tau
=\int_0^t\frac{2\pi}{A}+\fracd{}{t}\log A\,d\tau
=\int_0^t\frac{2\pi}{A}\,d\tau+\log\frac{A_t}{A_0}\,.\qedhere\]
\end{proof}

\begin{Prop}[Bonnesen isoperimetric inequality,~{\cite[Thm.~4\,(21)]{Osserman79}}]\label{prop:Osserman}
For an embedded, closed curve $\Si$ in the plane, 
$$\frac{L^2}{A}-4\pi\geq\frac{\pi^2}{A}(r_{\cir}-r_{\inner})^2\geq0\,,$$
where $r_{\cir}$ and $r_{\inner}$ are the circumscribed and inscribed radius of $\Si$.
\end{Prop}

\begin{Prop}
\label{prop:C0convergence}
Let $\Si_0=X_0(\Sp)$ be a smooth, embedded, convex curve.
Let $X:\Sp\!\times[0,\infty)\to\R^2$ be a solution of~\eqref{eq:ccf} with initial curve $\Si_0$ and $h$ satisfying~\eqref{eq:h_3a}.
Then 
$$r_{\cir}(t)-r_{\inner}(t)\to0$$
for $t\to\infty$ and $\St=X(\Sp\!,t)$ converges in $C^0$ to a circle of radius 
$$R:=\lim_{t\to\infty}\frac{L_t}{2\pi}=\lim_{t\to\infty}\sqrt{\frac{A_t}{\pi}}\in(0,\infty)\,.$$
Moreover, for all $\beta\in(0,1)$ there exist a time $t_0>0$ and a constant $C>0$ such that, for all $t\geq t_0$,
$$\left(\frac{L^2}{4\pi}-A\right)
	\leq C\exp\!\left(-\frac{2\beta t}{R^2}\right)
\quad\text{ and }\quad
\frac{L}{2\pi}\intst\kappa^2\,ds_t-2\pi
	\leq C\exp\!\left(-\frac{\beta t}{R^2}\right)\,.$$
\end{Prop}

\begin{proof}
By Lemma~\ref{lem:AL},
$$\int_0^t\frac{2\pi}{A}\,d\tau+\log\frac{A_t}{A_0}
\geq\frac{2\pi t}{C}+\log\frac{c}{A_0}
\to\infty$$
for $t\to\infty$.
Lemma~\ref{lem:isoleqexp} and the bounds from Lemma~\ref{lem:AL} imply
$$\frac{L}{2\pi}-\sqrt{\frac{A}\pi}\to0$$
for $t\to\infty$.
Also, $L/2\pi=\sqrt{A/\pi}$ only holds on a circle.
Proposition~\ref{prop:Osserman} yields $r_{\cir}(t)-r_{\inner}(t)\to0$ for $t\to\infty$.
Let $\beta\in(0,1)$ and $\vare(\beta,R)>0$ so that
$$\left(1-\vare R^2\right)\geq\beta\,.$$
We can choose $t_0(\beta)>0$ so that for all $t\geq t_0$,
$$\left(\frac1{R^2}-\vare\right)\leq\frac\pi A\,.$$
Hence,
$$-\int_0^t\frac{2\pi}{A}\,d\tau
\leq-2\left(1-\vare R^2\right)\frac{t}{R^2}
\leq-\frac{2\beta t}{R^2}$$
and again by the bounds on $A$ from Lemma~\ref{lem:AL},
\begin{align}\label{eq:L4piAexp}
\left(\frac{L^2}{4\pi}-A\right)
=\frac{A}{4\pi}\left(\frac{L^2}{A}-4\pi\right)
\leq C\exp\!\left(-\frac{2\beta t}{R^2}\right)
\end{align}
for all $t\geq t_0$.
Let $f\in C^2([0,\infty))$.
Since $C^2([0,\infty))$ is compactly embedded in $C^1([0,\infty))$ and $C^1([0,\infty))$ is continuously embedded in $C^0([0,\infty))$, Ehrling's Lemma yields that for every $\delta>0$ there exists $C(\delta)>0$ so that
$$\Vert f\Vert_{C^1([0,\infty))}\leq\delta\Vert f\Vert_{C^2([0,\infty))}+C(\delta)\Vert f\Vert_{C^0([0,\infty))}\,.$$
We set $\delta=1/2$ and conclude
\begin{align}\label{eq:dtfEhrlich}
\sup_{[0,\infty)}\left|\fracd{f}{t}\right|
\leq\sup_{[0,\infty)}\left|\fracd{^2f}{t^2}\right|+C\sup_{[0,\infty)}|f|\,.
\end{align}
Let $\eta>0$ and define $f_\eta:[0,\infty)\to\R$ by $f_\eta(t):=f(\eta t)$.
Then
$$\fracd{f_\eta}{t}(t)=\eta\fracd{f}{t}(\eta t)
\qquad\text{ and }\qquad
\fracd{^2f_\eta}{t^2}(t)=\eta^2\fracd{^2f}{t^2}(\eta t)$$
as well as with~\eqref{eq:dtfEhrlich},
\begin{align}\label{eq:dtfeta}
\sup_{[0,\infty)}\left|\fracd{f}{t}\right|
&=\frac1\eta\sup_{[0,\infty)}\left|\fracd{f_\eta}{t}\right|
\leq\frac1\eta\sup_{[0,\infty)}\left|\fracd{^2f_\eta}{t^2}\right|+\frac C\eta\sup_{[0,\infty)}|f_\eta| \notag\\
&=\eta\sup_{[0,\infty)}\left|\fracd{^2f}{t^2}\right|+\frac C\eta\sup_{[0,\infty)}|f|\,.
\end{align}
By Lemmata~\ref{lem:dthlp} and~\ref{lem:inthkappa}, there exists a time $t_1\geq t_0$ so that for all $t\geq t_1$,
\begin{align}\label{eq:L2piintkappaleqc}
\fracd{}{t}\left(\frac{L}{2\pi}\intst\kappa^2\,ds_t-2\pi\right)
\leq C\left(C+\frac{L}{2\pi}\intst\kappa^2\,ds_t-2\pi\right)^3
\leq C(C+1)^3\,.
\end{align}
We choose
$$\eta=\exp\!\left(-\frac{\beta t}{R^2}\right)$$
to obtain by~\eqref{eq:L2piintkappa},~\eqref{eq:L4piAexp},~\eqref{eq:dtfeta} and~\eqref{eq:L2piintkappaleqc}, for all $t\geq t_1$,
\begin{align*}
\frac{L}{2\pi}\intst\kappa^2\,ds_t-2\pi
&\leq\sup_{[t,\infty)}\left(\frac{L}{2\pi}\intst\kappa^2\,ds_t-2\pi\right) \\
&\leq\eta c\sup_{[t,\infty)}\left(1+\frac{L}{2\pi}\intst\kappa^2\,ds_t-2\pi\right)^3
	+\frac C\eta\sup_{[t,\infty)}\left|\frac{L^2}{4\pi}-A\right| \\
&\leq C\exp\!\left(-\frac{\beta t}{R^2}\right)\,.\qedhere
\end{align*}
\end{proof}

By Corollary~\ref{cor:strongmaxpkappa}, $\St$ is strictly convex for all $t>0$.
Like introduced in Section~\ref{sec:theta2} let $\vartheta:\Sp\!\times[0,\infty)\to\R$ be the angle between the $x_1$-axis and the tangent vector at the point $X(p,t)$.
Since $\St$ is strictly convex on $(0,\infty)$, $\vartheta(\,\cdot\,,t)$ is injective for each $t\in(0,\infty)$.
We want to use $\vartheta$ as spatial coordinate and define $\tau$ to be a new time variable so that $\tau=t$ as well as
\begin{align}\label{eq:taut}
\fracd{\tau}{t}=1\qquad\text{ and }\qquad\fracp{\vartheta}{\tau}=0\,.
\end{align}
The spatial derivatives transforms according to
$\frac1v\fracp{}{p}=\fracp{}{s}=\kappa\fracp{}{\vartheta}$.
In the following, we use the coordinates $(\vartheta,\tau)$ on $\Sp\!\times(0,\infty)$.

\begin{Lemma}[Gage--Hamilton~{\cite[Lem.~4.1.3]{GageHamilton86}} and Pihan~{\cite[Lem.~6.12]{Pihan98}}]\label{lem:dtaukappa}
Let $X:\Sp\!\times(0,\infty)\to\R^2$ be a smooth, strictly convex solution of~\eqref{eq:ccf}.
Then, for $\tau\in(0,\infty)$,
$$\fracp{\kappa}{\tau}=\kappa^2\fracp{^2\kappa}{\vartheta^2}-(h-\kappa)\kappa^2\,.$$
\end{Lemma}

For $\tau>0$, we define
\begin{align}\label{eq:defmtau}
m(\tau):=\max_{\bar\tau\in[0,\tau]}\max_{\vartheta\in\Sp}\kappa(\vartheta,\bar\tau)\,.
\end{align}

\begin{Lemma}
\label{lem:intdthetakappa}
Let $\Si_0=X_0(\Sp)$ be a smooth, embedded, convex curve.
Let $X:\Sp\!\times[0,\infty)\to\R^2$ be a solution of~\eqref{eq:ccf} with initial curve $\Si_0$ and $h$ satisfying~\eqref{eq:h_3a}. 
Then there exists a constant $C(\Si_0)>0$ such that, for all $\tau>0$,
$$\intsp\left(\fracp{\kappa}{\vartheta}\right)^{2}d\vartheta
\leq\intsp\kappa^2\,d\vartheta+C(m(\tau)+1)\,.$$
\end{Lemma}

\begin{proof}
We follow similar lines to~\cite[Lem.~3.4 and Cor.~3.5]{Gage86} and~\cite[Lem.~6.9]{Pihan98}.
We observe that
\begin{align}\label{eq:intkappakappamax}
0<\intst\kappa^2\,ds_t
=\intsp\kappa\,d\vartheta
\leq2\pi\max_{\vartheta\in\Sp}\kappa(\vartheta,\tau)
\end{align}
and use the time-independency~\eqref{eq:taut} of $\vartheta$, Lemma~\ref{lem:dtaukappa}, integration by parts to estimate 
\begin{align}\label{eq:dtauintkappadthetaap_1}
\fracd{}{\tau}\intsp\left(\kappa^2-\left(\fracp{\kappa}{\vartheta}\right)^{2}-2h\kappa\right)d\vartheta
&=\intsp2\left(\kappa+\fracp{^2\kappa}{^2\vartheta}-h\right)\fracp{\kappa}{\tau}\,d\vartheta
		-2\fracd{h}{\tau}\intsp\kappa\,d\vartheta \notag\\
&=2\intsp\kappa^2\left(\kappa+\fracp{^2\kappa}{^2\vartheta}-h\right)^{2}d\vartheta
		-2\fracd{h}{\tau}\intsp\kappa\,d\vartheta \notag\\
&\geq-2\fracd{h}{\tau}\intsp\kappa\,d\vartheta
\end{align}
for all $\tau>0$.
By~\eqref{eq:h_3a},
\begin{align}\label{eq:dtauhintkappa}
\fracd{h}{\tau}\intsp\kappa\,d\vartheta
=-(1-\gamma)\frac{2\pi}{L^2}\fracd{L}{\tau}\intsp\kappa\,d\vartheta
		+\frac{\gamma}{4\pi}\fracd{}{\tau}\left(\intsp\kappa\,d\vartheta\right)^2\,.
\end{align}
By Lemma~\ref{lem:AL}, $\fracd{}{\tau}L$ has a sign so that
$$\int_\vare^\tau\left|\fracd{L}{\bar\tau}\right|\,d\bar\tau\leq|L_\tau-L_0|\leq C$$
for all $0<\vare<\tau<\infty$. 
We integrate~\eqref{eq:dtauhintkappa} from $\vare$ to $\tau$ and conclude with $\vare\to0$, the upper bound from Lemma~\ref{lem:inthkappa}, the definition~\eqref{eq:defmtau} of $m$ and~\eqref{eq:intkappakappamax},
$$\int_0^\tau\fracd{h}{\bar\tau}\intsp\kappa\,d\vartheta\,d\bar\tau
\leq C\max_{\bar\tau\in[0,\tau]}\intsp\kappa\,d\vartheta
	+C\left(\intsp\kappa(\vartheta,\tau)\,d\vartheta\right)^2
\leq C(m(\tau)+1)$$
for all $\tau\in(0,\infty)$.
Hence, integrating~\eqref{eq:dtauintkappadthetaap_1} and the bounds from Lemma~\ref{lem:inthkappa} yield the claim. 
\end{proof}

For $\tau>0$, define
\begin{align}\label{eq:defmstartau}
m^*(\tau):=1+\frac{\sqrt{m(\tau)+1}}{\max_{\vartheta\in\Sp}\kappa(\vartheta,\tau)}\,.
\end{align}

\begin{Lemma}\label{lem:varekappamax}
Let $\Si_0=X_0(\Sp)$ be a smooth, embedded, convex curve.
Let $X:\Sp\!\times[0,\infty)\to\R^2$ be a solution of~\eqref{eq:ccf} with initial curve $\Si_0$ and $h$ satisfying~\eqref{eq:h_3a}.
Let $\tau>0$, $\vartheta_1,\vartheta_2\in\Sp$ and $\delta\in(0,\pi/2]$.
If $|\vartheta_1-\vartheta_2|<\delta$, then there exists $C\geq\sqrt{2\pi}$ with
$$|\kappa(\vartheta_1,\tau)-\kappa(\vartheta_2,\tau)|<Cm^*(\tau)\sqrt\delta\max_{\vartheta\in\Sp}\kappa(\vartheta,\tau)\,.$$
\end{Lemma}

\begin{proof}
We follow similar lines to~\cite[Paragraph~4.3.6]{GageHamilton86} and~\cite[Lem.~7.1]{Pihan98}.
Lemma~\ref{lem:AL} provides
$$\max_{\vartheta\in\Sp}\kappa(\vartheta,\tau)\geq\frac{L}{2\pi}\geq c>0\,.$$
Let $\delta\in(0,\pi/2]$.
For $|\vartheta_1-\vartheta_2|<\delta$, Cauchy--Schwarz and Lemma~\ref{lem:intdthetakappa} imply
\begin{align*}
|\kappa(\vartheta_1,\tau)-\kappa(\vartheta_2,\tau)|
&\leq|\vartheta_1-\vartheta_2|^{1/2}\left(\int_{\vartheta_1}^{\vartheta_2}
		\left(\fracp{\kappa}{\vartheta}(\vartheta,\tau)\right)^{2}d\vartheta\right)^{1/2} \notag\\
&\leq\sqrt{\delta}\left(\intsp\kappa^2(\vartheta,\tau)\,d\vartheta+C(m(\tau)+1)\right)^{1/2} \notag\\
&\leq\sqrt{\delta}\left(\sqrt{2\pi}\max_{\vartheta\in\Sp}\kappa(\vartheta,\tau)+\sqrt{C(m(\tau)+1)}\right) \notag\\
&\leq\sqrt{\delta}C\max_{\vartheta\in\Sp}\kappa(\vartheta,\tau)\left(1+\frac{\sqrt{m(\tau)+1}}{\max_{\vartheta\in\Sp}\kappa(\vartheta,\tau)}\right)\,,
\end{align*}
where we used $\max_{\vartheta\in\Sp}\kappa(\vartheta,\tau)>0$ for $\tau>0$.
\end{proof}

\begin{Lemma}[Gage--Hamilton~{\cite[Cor.~5.2]{GageHamilton86}}]\label{lem:kappamaxrin}
Let $\Si_0=X_0(\Sp)$ be a smooth, embedded, convex curve.
Let $X:\Sp\!\times[0,\infty)\to\R^2$ be a smooth, embedded solution of~\eqref{eq:ccf} with initial curve $\Si_0$ and $h$ satisfying~\eqref{eq:h_3a}. 
Let $\vare\in(0,1)$ and $\tau>0$.
Then
$$\max_{\vartheta\in\Sp}\kappa(\vartheta,\tau)r_{\inner}(\tau)
\leq\left\{(1-\vare)\left[1-K\!\left(\left(\frac\vare{Cm^*(\tau)}\right)^2\right)\left(\frac{r_{\cir}(\tau)}{r_{\inner}(\tau)}-1\right)\right]\right\}^{-1}\,,$$
where $K:(0,\pi]\to[0,\infty)$ is a positive decreasing function with $K(\omega)\to\infty$ for $\omega\searrow0$ and $K(\pi)=0$.
\end{Lemma}

\begin{proof}
The proof follows with the help of Lemma~\ref{lem:varekappamax} and can be found in~\cite[Cor.~5.2]{GageHamilton86} and~\cite[Lem.~7.11]{Pihan98}.
For details, see also~\cite[Lem.~11.11]{Dittberner18}.
\end{proof}

\begin{Cor}\label{cor:kappamaxrinvare}
Let $\Si_0=X_0(\Sp)$ be a smooth, embedded, convex curve.
Let $X:\Sp\!\times[0,\infty)\to\R^2$ be a smooth, embedded solution of~\eqref{eq:ccf} with initial curve $\Si_0$ and $h$ satisfying~\eqref{eq:h_3a}.
For every $\vare\in(0,1)$, there exists a time $\tau_0>0$ such that, for all $\tau\geq\tau_0$,
$$\max_{\vartheta\in\Sp}\kappa(\vartheta,\tau)r_{\inner}(\tau)\leq\frac1{(1-\vare)^2}\,.$$
\end{Cor}

\begin{proof}
We extend the proof of~\cite[Prop.~5.3]{GageHamilton86} and~\cite[Cor.~7.12]{Pihan98}.
Proposition~\ref{prop:C0convergence} implies that, for every $\delta>0$, there exists a time $\tau_0(\delta)>0$ so that $r_{\cir}(\tau)-r_{\inner}(\tau)\leq\delta$ for all $\tau\geq\tau_0$, and thus
\begin{align}\label{eq:routrindelta}
\frac{r_{\cir}(\tau)}{r_{\inner}(\tau)}-1\leq\frac\delta{r_{\inner}(\tau)}\,.
\end{align}
Recall the definitions~\eqref{eq:defmtau} and~\eqref{eq:defmstartau} of $m$ and $m^*$.
We define
$$I_1:=\{\tau\geq\tau_0\;|\;m(\tau)=\max_{\vartheta\in\Sp}\kappa(\vartheta,\tau)\}$$
and
$$I_2:=\{\tau\geq\tau_0\;|\;m(\tau)>\max_{\vartheta\in\Sp}\kappa(\vartheta,\tau)\}\,.$$
Then, $m$ is monotonically increasing on $I_1$ and constant on every connected subinterval of $I_2$.
By Lemma~\ref{lem:AL}, $L$ is uniformly bounded from above.
Hence there exits a constant $c>0$ so that $\max_{\vartheta\in\Sp}\kappa(\vartheta,\tau)\geq c$ for all $\tau\in[\tau_0,\infty)$ and
$$m^*(\tau)\leq2+\frac1c$$
for $\tau\in I_1$.
We distinguish between three cases. 
\begin{enumerate}[(i)]
\item Assume that $\sup_{[\tau_0,\infty)}m<\infty$.
Then $\sup_{[\tau_0,\infty)}m^*<\infty$.
\item Assume that $\sup_{[\tau_0,\infty)}m=\infty$ and $\sup\{\tau\in I_2\}=:\tau_1<\infty$.
Then $[\tau_1,\infty)\subset I_1$ and 
$$\sup_{[\tau_0,\infty)}m^*=\sup_{I_1}m^*<2+\frac1c\,.$$
\item Assume that $\sup_{[\tau_0,\infty)}m=\infty$ and $\sup\{\tau\in I_2\}=\infty$.
Assume there exists $\tau_2\in[\tau_0,\infty)$ so that $(\tau_2,\infty)\subset I_2$, then $m(\tau)=m(\tau_2)<\infty$ for all $\tau\in(\tau_2,\infty)$. 
This contradicts $\sup_{[\tau_0,\infty)}m=\infty$. 
Hence, $I_2$ consists of infinitely many disjoint open intervals $I_{2,k}$, $k\in\N$ and $\sup_{I_1}m^*\leq2+1/c$.
Define the sequence
$$\big(\tau_k:=\sup\{\tau\in I_{2,k}\}\in I_1\big)_{k\in\N}\,.$$
Then $\tau_k\to\infty$ for $k\to\infty$ and for all $k\in\N$, and since $\tau_k\in I_1$,
$$m(\tau)=m(\tau_k)=\max_{\vartheta\in\Sp}\kappa(\vartheta,\tau_k)$$
as well as
$$m^*(\tau)\leq1+\frac{m(\tau)+1}{\max_{\vartheta\in\Sp}\kappa(\vartheta,\tau)}
=1+\frac{m(\tau_k)+1}{\max_{\vartheta\in\Sp}\kappa(\vartheta,\tau_k)}
\leq2+\frac1c$$
for all $\tau\in I_{2,k}$.
Hence, 
$$\sup_{\tau\in[\tau_0,\infty)}m(\tau)
=\sup_{\tau\in I_1\cup I_2}m(\tau)
\leq2+\frac1c\,.$$
\end{enumerate}
Thus, for any $\tau\geq\tau_0$, $m^*$ is independent of time.
Recall that $K$, as defined in Lemma~\ref{lem:kappamaxrin}, is a positive decreasing function that satisfies $K(\omega)\to\infty$ for $\omega\searrow0$ and $K(\pi)=0$.
By Proposition~\ref{prop:C0convergence}, $r_{\inner}(\tau)\geq c>0$ for all $\tau\geq0$.
Hence, for given $\vare\in(0,1)$, we can choose $\delta>0$ and $\tau_0(\delta)>0$ so that
\begin{align}\label{eq:rinKvare}
\frac\delta{r_{\inner}(\tau)}\leq\frac\vare{K\big((\vare/Cm^*)^2\big)}
\end{align}
for all $\tau\geq\tau_0$.
Combining~\eqref{eq:routrindelta} and~\eqref{eq:rinKvare} yields
$$\frac{r_{\cir}(\tau)}{r_{\inner}(\tau)}-1\leq\frac\vare{K\big((\vare/Cm^*)^2\big)}$$
so that
$$1-\vare\leq1-K\!\left(\left(\frac\vare{Cm^*}\right)^2\right)\left(\frac{r_{\cir}(\tau)}{r_{\inner}(\tau)}-1\right)$$
for all $\tau\geq\tau_0$.
This and Lemma~\ref{lem:kappamaxrin} imply 
$$\max_{\vartheta\in\Sp}\kappa(\vartheta,\tau)r_{\inner}(\tau)
\leq\frac1{(1-\vare)^2}$$
for any $\tau\geq\tau_0$.
\end{proof}

\begin{Cor}
\label{cor:C2congergences}
Let $\Si_0=X_0(\Sp)$ be a smooth, embedded, convex curve.
Let $X:\Sp\!\times[0,\infty)\to\R^2$ be a smooth, embedded solution of~\eqref{eq:ccf} with initial curve $\Si_0$ and $h$ satisfying~\eqref{eq:h_3a}.
Then 
$$\frac{\max_{\vartheta\in\Sp}\kappa(\vartheta,\tau)}{\min_{\vartheta\in\Sp}\kappa(\vartheta,\tau)}\to1\,,\qquad
\kappa(\vartheta,\tau)\to\frac1{R}\qquad\text{ and }\qquad
h(\tau)\to\frac1{R}$$
for every $\vartheta\in\Sp$ and for $\tau\to\infty$, where $R$ is given in Proposition~\ref{prop:C0convergence}.
\end{Cor}

\begin{proof}
We follow the lines of~{\cite[Cor.~7.14]{Pihan98}}.
By Proposition~\ref{prop:C0convergence}, $\Si_\tau$ is strictly convex for $\tau\in(0,\infty)$.
Like in~\cite[Thm.~5.4]{GageHamilton86},~\cite[Prop.~7.13]{Pihan98} or~\cite[Prop.~11.13]{Dittberner18}, we first conclude with the help of Corollary~\ref{cor:kappamaxrinvare} that $\kappa(\vartheta,\tau)r_{\inner}(\tau)\to1$ for all $\vartheta\in\Sp$ and for $\tau\to\infty$.
Hence, it also holds that $\max_{\vartheta\in\Sp}\kappa(\vartheta,\tau)r_{\inner}(\tau)\to1$ and $\min_{\vartheta\in\Sp}\kappa(\vartheta,\tau)r_{\inner}(\tau)\to1$ for $\tau\to\infty$ and the first claim follows.
By Proposition~\ref{prop:C0convergence}, the curve converges to a circle of radius $R$. 
This yields the second claim. 
The third claim follows from Lemma~\ref{lem:inthkappa} and $L\to2\pi R$.
\end{proof}

\begin{Thm}[Pihan~{\cite[Prop.~7.17]{Pihan98}}]
\label{thm:kappanconvergence} 
Let $\Si_0=X_0(\Sp)$ be a smooth, embedded, convex curve.
Let $X:\Sp\!\times[0,\infty)\to\R^2$ be a smooth, embedded solution of~\eqref{eq:ccf} with initial curve $\Si_0$ and $h$ satisfying~\eqref{eq:h_3a}. 
Then, for all $n\in\N$, $\fracp{^n}{\vartheta^n}\kappa\to0$ uniformly for $\tau\to\infty$.
Hence, the curves converge uniformly in $C^\infty$ to a circle of radius $R$.
\end{Thm}

\begin{proof}
The proof uses Corollary~\ref{cor:C2congergences} and can be found in~\cite[Prop.~7.17]{Pihan98} or~\cite[Thm.~11.17]{Dittberner18}.
\end{proof}

We summarise our results in the following and two theorems.

\begin{Thm}
\label{thm:main1}
Let $\Si_0=X_0(\Sp)$ be a smooth, embedded curve.
Let $X:\Sp\!\times[0,\infty)\to\R^2$ be a smooth, embedded solution of~\eqref{eq:ccf} with initial curve $\Si_0$ and $h$ satisfying~\eqref{eq:h_3a}. 
Then the evolving surfaces $\St=X(\Sp\!,t)$ are contained in a uniformly bounded region of the plane for all times.
And, for all $\beta\in(0,1)$, there exists a time-independent constant $C>0$ such that, for all $t\geq0$,
\begin{enumerate}[(i)]
\item $|\max_{p\in\Sp}\kappa(p,t)-\min_{p\in\Sp}\kappa(p,t)|\leq C\exp\!\left(-\frac{\beta}{R^2}t\right)$,
\item $|\kappa(p,t)-1/R|\leq C\exp\!\left(-\frac{\beta}{R^2}t\right)$ for all $p\in\Sp$,
\item $|h(t)-1/R|\leq C\exp\!\left(-\frac{\beta}{R^2}t\right)$, and
\item $\left|\fracp{^n}{t^m}\fracp{^n}{p^n}\kappa(p,t)\right|\leq C\exp\!\left(-\frac{\beta}{(n+2m+1)R^2}t\right)$ for all $p\in\Sp$ and all $n,m\in\N$.
\end{enumerate}
Hence, the solution converges smoothly and exponentially to a circle of radius $R$.
\end{Thm}

\begin{proof}
By Theorem~\ref{thm:convexity}, there exists a time $T_0>0$ so that the curves are strictly convex on $(T_0,\infty)$.
Like in~\cite{GageHamilton86},~\cite[Section~7.5]{Pihan98} and~\cite[Section~11.4]{Dittberner18}, we can show for convex curves with the help of Wirtinger's inequality and the smooth convergence of Theorem~\ref{thm:kappanconvergence} exponential decay of the $L^2$-norm of the derivative of the curvature.
The proof is independent of the particular form of $h$, which is why we do not repeat it here.
Interpolation inequalities then yield that for $\beta\in(0,1)$ and $m,n\in\N\cup\{0\}$, $m+n>0$, there exist constants $C_{n,m}>0$ such that
\begin{align}\label{eq:dtaundthetamkappa}
\max_{\vartheta\in\Sp}\left|\fracp{^m}{\tau^m}\fracp{^n\kappa}{\vartheta^n}(\vartheta,\tau)\right|
\leq C_{n,m}\exp\!\left(-\frac{\beta\tau}{(n+2m+1)R^2}\right)
\end{align}
for $\tau$ large enough.
To prove~(i), we follow the lines of~\cite[Prop.~7.27]{Pihan98}.
For $t\geq0$, let $p_1,p_2\in\Sp$ be the points where the curvature attains its maximum and minimum.
By Lemma~\ref{lem:AL} and~\eqref{eq:dtaundthetamkappa}, there exists a time-independent constant $C>0$ so that
$$\left|\max_{p\in\Sp}\kappa(p,t)-\min_{p\in\Sp}\kappa(p,t)\right|
=|\kappa(p_2,t)-\kappa(p_1,t)|
\leq\intst\left|\fracp{\kappa}{s}\right|ds_t
\leq C\exp\!\left(-\frac{\beta t}{R^2}\right)\,.$$
for $t$ large enough. 
To show claim (ii), we observe that for embedded, closed, convex curves,
\begin{align}\label{eq:kappaminALkappamax}
\min_{p\in\Sp}\kappa(p)
\leq\frac1{r_{\cir}}
=\sqrt{\frac\pi{A_{\cir}}}
\leq\sqrt{\frac\pi{A}}
\leq\frac{2\pi}{L}
\leq\frac1{2\pi}\intst\kappa^2\,ds_t
\leq\max_{p\in\Sp}\kappa(p)\,.
\end{align}
By the intermediate value theorem and~\eqref{eq:kappaminALkappamax} there exist points $p_0,p_1,p_2\in\Sp$ with $\kappa(p_0,t)=\sqrt{\pi/A}$, $\kappa(p_1,t)=2\pi/L$ and $\kappa(p_2,t)=\intst\kappa^2\,ds_t/2\pi$, so that for $p\in\Sp$ with~\eqref{eq:dtaundthetamkappa},
\begin{align}\label{eq:kappapiA}
\left|\kappa(p,t)-\sqrt{\frac\pi{A}}\right|
=|\kappa(p,t)-\kappa(p_0,t)|
\leq\intst\left|\fracp{\kappa}{s}\right|ds_t
\leq C\exp\!\left(-\frac{\beta t}{R^2}\right)
\end{align}
and likewise
\begin{align}\label{eq:kappa2piL}
\left|\kappa(p,t)-\frac{2\pi}{L}\right|
+\left|\kappa(p,t)-\frac1{2\pi}\intst\kappa^2\,ds_t\right|
\leq C\exp\!\left(-\frac{\beta t}{R^2}\right)
\end{align}
for $t$ large enough.
Furthermore, Proposition~\ref{prop:C0convergence} and yields
\begin{align}\label{eq:AR}
\left|\sqrt{\frac\pi{A}}-\frac1R\right|
\leq\frac{\sqrt{|A-\pi R^2|}}{R\sqrt{A}}
\leq C\sqrt{\int_t^\infty\left|\fracd{A}{\tau}\right|\,d\tau}
\leq C\int_t^\infty\left|\fracd{A}{\tau}\right|\,d\tau
\end{align}
and
\begin{align}\label{eq:LR}
\left|\frac{2\pi}{L}-\frac1R\right|
=\frac{|2\pi R-L|}{LR}
\leq C\int_t^\infty\left|\fracd{L}{\tau}\right|\,d\tau\,.
\end{align} 
By~\eqref{eq:L2piintkappa},~\eqref{eq:dtAa} and Proposition~\ref{prop:C0convergence}, there exists a constant $C>0$ so that
\begin{align}\label{eq:piAR_a}
\sqrt{\int_t^\infty\left|\fracd{A}{\tau}\right|\,d\tau}
&=\sqrt{\gamma\left(\frac{L^2}{4\pi}-A\right)}
\leq C\exp\!\left(-\frac{\beta t}{R^2}\right)
\end{align}
for $t$ large enough.
By~\eqref{eq:kappapiA},~\eqref{eq:kappa2piL},~\eqref{eq:AR},~\eqref{eq:LR} and~\eqref{eq:piAR_a}, for $p\in\Sp$,
\begin{align}\label{eq:kappa1Rleq}
\left|\kappa(p,t)-\frac1{R}\right|
\leq\left|\kappa(p,t)-\sqrt{\frac\pi{A}}\right|+\left|\sqrt{\frac\pi{A}}-\frac1{R}\right|
\leq C\exp\!\left(-\frac{\beta t}{R^2}\right)\,.
\end{align}
Likewise with Proposition~\ref{prop:C0convergence},~\eqref{eq:kappapiA} and~\eqref{eq:kappa2piL},
\begin{align}\label{eq:kappahleq}
|\kappa(p,t)-h(t)|
&\leq\left|\kappa(p,t)-\frac{2\pi}L\right|+\frac{|\gamma|}{L}\left|\frac L{2\pi}\intst\kappa^2\,ds_t-2\pi\right| \notag\\
&\leq C\exp\!\left(-\frac{\beta t}{R^2}\right)
\end{align}
for $t$ large enough.
The boundedness of the curvature on $[0,T_0]$ yields the claim for all $t\geq0$.
For claim (iii), we estimate with~\eqref{eq:kappa1Rleq} and~\eqref{eq:kappahleq},
$$\left|h-\frac1{R}\right|
\leq C\exp\!\left(-\frac{\beta t}{R^2}\right)$$
for all $t\geq0$.
For claim~(iv), we use Lemma~\ref{lem:evolutionequations} and~\eqref{eq:kappahleq} to estimate
$$\fracp{v}{t}
=\kappa(h-\kappa)v
\leq C\left(\exp\!\left(-\frac{\beta t}{R^2}\right)\right)v$$
for all $t\geq0$.
Hence, $v\geq C$ on $\Sp\!\times[0,\infty)$ and the claim follows with~\eqref{eq:dtaundthetamkappa}, for every $m,n\in\N\cup\{0\}$, $m+n>0$.
To show that the curves stay in a bounded region, we observe that with~\eqref{eq:kappahleq},
$$\Vert X(p,t)-X(p,0)\Vert_{\R^2}
\leq\int_0^t|\kappa(p,\tau)-h(\tau)|\,d\tau
\leq C\int_0^t\exp\!\left(-\frac{\beta\tau}{R^2_0}\right)\,d\tau\leq C$$
for all $p\in\Sp$ and $t\in(0,\infty)$, where $C$ is independent of time.
\end{proof}

\begin{Rem}\label{rem:dtAg}
All the proofs leading up to Theorem~\ref{thm:main1} also work, if we prescribe the derivative of the area or the length by a function $g\in C^\infty([0,\infty))\cap L^1([0,\infty))$. 
If we prescribe the derivative of the area, Lemma~\ref{lem:dtL} and~\eqref{eq:hdtAdtL} yield
\begin{align}\label{eq:dtAdtLb}
\fracd{A}{t}=g\,,\quad
h=\frac{2\pi+g}L
\quad\text{ and }\quad
\fracd{L}{t}
=-\frac{2\pi}L\left(\frac L{2\pi}\intst\kappa^2\,ds_t-2\pi-g\right)\,,
\end{align}
where either
\begin{align*}
-2\pi<g\leq0\,,&\quad\fracd{g}{t}\geq0\quad\text{ and }\quad\int_0^\infty g\,dt>-A_0\,,\qquad\text{ or } \\
0\leq g<\frac{L_t}{2\pi}\intst\kappa^2\,ds_t-2\pi\,,&\quad\fracd{g}{t}\leq0\quad\text{ and }\quad\int_0^\infty g\,dt\leq\frac{L_0^2}{4\pi}-A_0\,,
\end{align*}
since need $A$ and $L$ to be monotone and bounded and we will need $h$ to be positive in Remark~\ref{rem:main2}.
If we prescribe the derivative of the length, Lemma~\ref{lem:dtL} and~\eqref{eq:hdtAdtL} yield
\begin{align*}
\fracd{L}{t}=g\,,\quad
h=\frac1{2\pi}\left(\intst\kappa^2\,ds_t+g\right)
\quad\text{ and }\quad
\fracd{A}{t}
=\frac L{2\pi}\intst\kappa^2\,ds_t-2\pi+\frac{Lg}{2\pi}\,,
\end{align*}
where either
\begin{align*}
-\intst\kappa^2\,ds_t+\frac{4\pi}{L_t}<g\leq0\,,&\quad\fracd{g}{t}\geq0\quad\text{ and }\;\int_0^\infty g\,dt>-L_0\,,\quad\text{ or } \\
0\leq g\,,&\quad\fracd{g}{t}\leq0\quad\text{ and }\quad\int_0^\infty g\,dt<\infty\,,
\end{align*}
since again need $A$ and $L$ to be monotone and bounded.
Then Theorem~\ref{thm:main1} holds with the addition in the cases
\begin{enumerate}
\item[(ii)] $|\kappa(p,t)-1/R|\leq C\exp\!\left(-\frac{\beta}{R^2}t\right)+C\int_t^\infty g\,d\tau$ for all $p\in\Sp$, and
\item[(iii)] $|h(t)-1/R|\leq C\exp\!\left(-\frac{\beta}{R^2}t\right)+C\int_t^\infty g\,d\tau+Cg(t)$
\end{enumerate}
for all $\beta\in(0,1)$ and $t\geq0$, where $C>0$ is time-independent.
\end{Rem}

\begin{Thm}\label{thm:main2}
Let $\Si_0=X_0(\Sp)$ be a smooth, embedded curve satisfying~\eqref{eq:intkappageqminuspi}.
Then there exists a unique, smooth, embedded solution $X:\Sp\!\times[0,\infty)\to\R^2$ to~\eqref{eq:ccf} with initial curve $\Si_0$ and $h$ satisfying~\eqref{eq:h_ap}.
The evolving curves $\St=X(\Sp\!,t)$ are contained in a uniformly bounded region and converge smoothly and exponentially to a circle of radius $R$.
\end{Thm}

\begin{proof}
By the short time existence, there exists a unique solution $X\in C^\infty(\Sp\!\times[0,T))$ 
By Lemma~\ref{lem:AL}, $c\leq L\leq C$ so that $h$ is uniformly bounded from above and below away from zero.
By Corollary~\ref{cor:embeddedness} the curves remain embedded on $(0,T)$.
Corollary~\ref{cor:T=infty} yields that $T=\infty$.
Hence, we can apply Theorem~\ref{thm:main1}.
\end{proof}

\begin{Rem}\label{rem:main2}
Theorem~\ref{thm:main2} also holds for $h$ satisfying~\eqref{eq:dtAdtLb}.
\end{Rem}

\providecommand{\bysame}{\leavevmode\hbox to3em{\hrulefill}\thinspace}
\providecommand{\MR}{\relax\ifhmode\unskip\space\fi MR }
\providecommand{\MRhref}[2]{%
  \href{http://www.ams.org/mathscinet-getitem?mr=#1}{#2}
}
\providecommand{\href}[2]{#2}

\bibliographystyle{amsplain} 
\end{document}